\documentclass[12pt]{article}

\usepackage{amsmath}
\usepackage{amsthm}
\usepackage{amssymb}
\usepackage{cite}
\usepackage{url}
\usepackage[multiple]{footmisc}
\usepackage{graphicx}
\usepackage{epstopdf}
\usepackage{chngcntr}
\usepackage{enumitem}
\usepackage{hhline}
\usepackage{array}
\usepackage{hyperref}
\usepackage{color}

\hypersetup{colorlinks=false}

\if@twoside
\else 
\fi

\numberwithin{equation}{section}
\counterwithin{figure}{section}
\counterwithin{table}{section}

\theoremstyle{plain}
\newtheorem{theorem}{Theorem}[section]

\newtheorem{proposition}[theorem]{Proposition}

\newtheorem{problem}{Problem}

\theoremstyle{definition}

\theoremstyle{remark}

\newcommand{\norm}[1]{\left\|#1\right\|}
\newcommand{\abs}[1]{\left\vert#1\right\vert}
\newcommand{\spr}[1]{\left\langle\,#1\,\right\rangle}
\newcommand{\kl}[1]{\left(#1\right)}
\newcommand{\Kl}[1]{\left\{#1\right\}}

\definecolor{aoe}{rgb}{0.0, 0.5, 0.0}

\newcommand{\R}{\mathbb{R}}

\newcommand{\yd}{y^{\delta}}

\newcommand{\Ein}{E^{\text{in}}}
\newcommand{\Eout}{E^{\text{out}}}
\newcommand{\Eref}{E_\text{ref}}

\newcommand{\Pref}{P_\text{ref}}
\newcommand{\Iref}{I_\text{ref}}

\newcommand{\DT}{\Delta T}

\newcommand{\OD }{{\Omega_D}}
\newcommand{\OS}{{\Omega_S}}

\newcommand{\Lt}{{L_2}}
\newcommand{\LtOD}{{L_2(\OD)}}
\newcommand{\LtOS}{{L_2(\OS)}}
\newcommand{\LoR}{{L_1(\R)}}
\newcommand{\LtR}{{L_2(\R)}}
\newcommand{\LiR}{{L_\infty(\R)}}

\newcommand{\Linf}{{L_\infty}}

\newcommand{\chisupf}{\chi_{\sup(f)}}

\newcommand{\sitj}{{s_i,\theta_j}}
\renewcommand{\ij}{{i,j}}

\newcommand{\E}{\mathcal{E}}

\newcommand{\foh}{\frac{1}{2}}

\newcommand{\vu}{\boldsymbol{u}}

\title{A mathematical approach towards THz tomography for non-destructive imaging}
\author{
Simon Hubmer\footnote{Johann Radon Institute Linz, Altenbergerstra{\ss}e 69, A-4040 Linz, Austria, (simon.hubmer@ricam.oeaw.ac.at), Corresponding author.},
Alexander Ploier\footnote{Doctoral Program Computational Mathematics, Altenbergerstra{\ss}e 69, A-4040 Linz, Austria, (alexander.ploier@dk-compmath.jku.at)},
Ronny Ramlau\footnote{Johannes Kepler University Linz, Institute of Industrial Mathematics, Altenbergerstra{\ss}e 69, A-4040 Linz, Austria, (ronny.ramlau@jku.at)} \footnote{Johann Radon Institute Linz, Altenbergerstra{\ss}e 69, A-4040 Linz, Austria, (ronny.ramlau@ricam.oeaw.ac.at)}, \\
Peter Fosodeder\footnote{Research Center for Non Destructive Testing GmbH (RECENDT), Altenbergerstra{\ss}e 69, A-4040 Linz, Austria, (peter.fosodeder@recendt.at)}, 
Sandrine van Frank\footnote{Research Center for Non Destructive Testing GmbH (RECENDT), Altenbergerstra{\ss}e 69, A-4040 Linz, Austria, (sandrine.vanfrank@recendt.at)}
}

\begin{document}

\maketitle

\begin{abstract}

In this paper, we consider the imaging problem of terahertz (THz) tomography, in particular as it appears in non-destructive testing.  We derive a nonlinear mathematical model describing a full THz tomography experiment, and consider linear approximations connecting THz tomography with standard computerized  tomography and the Radon transform. Based on the derived models we propose different reconstruction approaches for solving the THz tomography problem, which we then compare on experimental data obtained from THz measurements of a plastic sample. 

\end{abstract}


\section{Introduction}

The plastics industry is a major player in the global economy and has an important role to play in the ecological transition into an energy-efficient and waste-free society. In addition to a higher recycling rate, a better design and more reliable methods for the production of plastic products are critical to reach the Sustainable Development Goals.

In this work, we particularly focus on the quality control of extruded plastic profiles. Presently, only simple geometries like pipes and thin layers can be measured in a safe way with inline sensors \cite{Yahyapour_Jahn_ThicknesswECOPS_2019}. Often, only single points along a profile are evaluated in this way, while more complex geometries are either tested in a destructive way, which results in long feedback times for process control, or with unsafe radiation \cite{Garcea_Wang_Withers_XCTforComposites_2018}. In the latter case, bulky and cost-intensive X-ray scanners are used, which is often undesirable due to human safety concerns and costs. Here, we try to improve the current state-of-the-art by utilizing Terahertz (THz) radiation for an inline tomography measurement on plastic profiles. Since typical polymers exhibit low absorption in a range of frequencies up to several THz \cite{Jin_Kim_Jeon_THzDielectricConstants_2006}, this type of radiation is a highly suitable candidate for performing non-destructive testing in a safe way. Furthermore, newly developed source and detector technology \cite{Dietz_Vieweg_ECOPS_2014} opens up the possibility of designing relatively compact, cheap, fast and robust measurement devices for the use in industrial environments.

Typical THz systems can be classified as continuous wave (cw) or pulsed measurement devices. Depending on the actual use case, both cw THz-CT \cite{Recur_cwTHzCT_2012} and pulsed THz-CT \cite{Mukherjee_Federici_pulsedTHzCT} have been used in the past. While cw systems generally offer more radiation power at relatively low frequencies of several 100 GHz, pulsed systems enable one to directly measure the electric field of THz signals with frequencies up to several THz. In particular, the simultaneous accessibility of amplitude and phase information in the measured pulse opens up a new path for data evaluation and information extraction. In combination with the better diffraction-limited resolution associated with the larger radiation frequency, this advantage convinced us to use a pulsed THz Time-Domain Spectroscopy \cite{Neu_Schmuttenmaer_THzTDS_2018} (THz-TDS) system for our work. 

In this paper, we develop a new imaging model specifically for the use case of THz imaging on plastic profiles, in order to comprehensively make use of the information contained in the measurement signal of the THz-TDS system. Based on physical considerations, our derivation results in a nonlinear model involving the object density function and the Radon transform. Furthermore, we derive linear approximations of this model revealing connections to standard computerized (X-ray) tomography. Based on the derived models, we then consider a number of linear and nonlinear reconstruction approaches for solving the THz tomography problem, and compare their performance on experimental data obtained from Thz measurements of a plastic sample.

The outline of this paper is as follows: In Section~\ref{sect_background} we give some physical background on THz tomography, and discuss the experimental setup which we work with. In Section~\ref{sect_model} we derive mathematical models describing that setup, which then forms the basis of the numerical reconstruction procedures presented in Section~\ref{sect_reconstruction}. The application of the presented methods to experimental data is considered in Section~\ref{sect_numerical_results}, and is followed by a short conclusion given in Section~\ref{sect_conclusion}.

\section{Physical background}\label{sect_background}

In this section, we outline the THz CT measurement setup and postulate a THz imaging model based on geometrical optics. Starting from Maxwells equations, several approximations will be introduced, until the final imaging model can be expressed in the form of five simple statements.

\subsection{Measurement Setup and Procedure}\label{subsect_Setup}

The schematic drawing in Figure~\ref{fig_ImagingSetup} shows the imaging part of the experimental setup. It shall be noted, that several other components, mostly those relevant for generating and detecting THz radiation, are omitted as they are not directly relevant for understanding the imaging procedure. 

A photoconductive antenna (PCA) is used to transmit THz radiation (Tx in Figure~\ref{fig_ImagingSetup}). By means of two off-axis parabolic mirrors (OPM) the beam of THz radiation is firstly collimated and secondly focussed into the imaging region. The object of interest is mounted on mechanical stages for performing rotary and translatory movement in the focal plane. After interacting with the object, the THz radiation is analogously guided and focussed into a receiving PCA (Rx in Figure~\ref{fig_ImagingSetup}). 

\begin{figure}[h!]
    \centering
    \includegraphics[viewport=0bp 560bp 590bp 790bp,clip,scale=0.7]{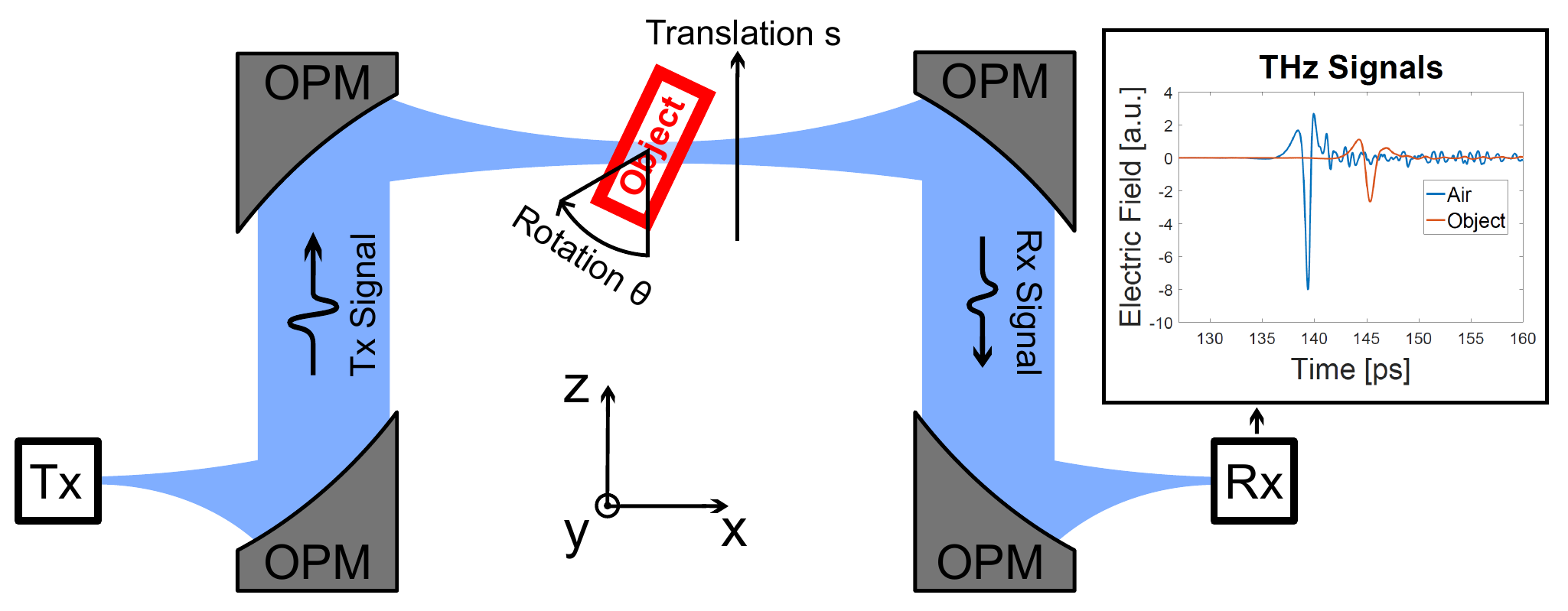}\\
    \includegraphics[scale=0.16]{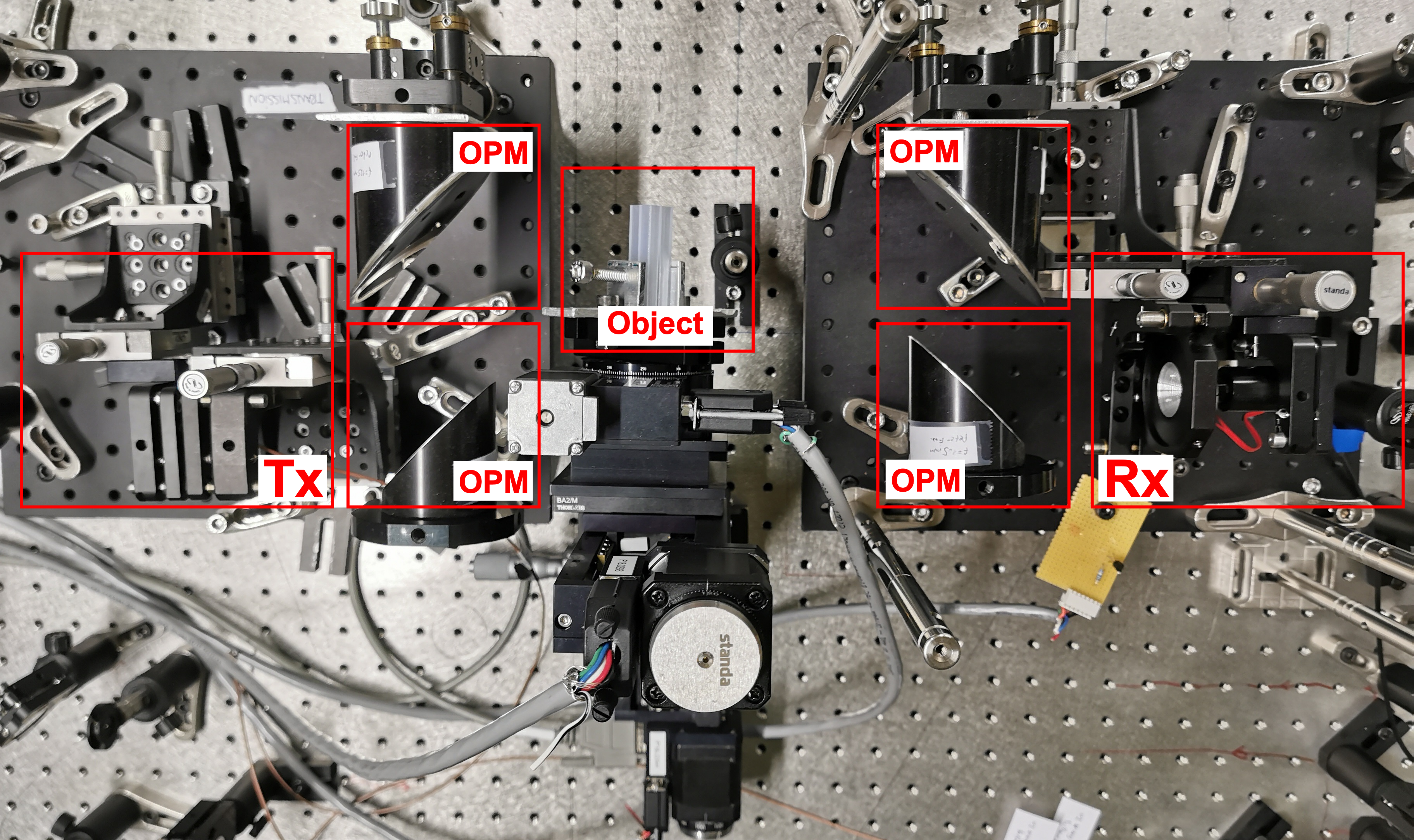}
    \caption{Schematic drawing and image of the measurement setup. The THz radiation is generated by a transmitting antenna (Tx). Two Off-Axis Parabolic Mirrors (OPM) are used to create a focussed THz beam. After interacting with an object, the THz beam is guided to the detecting antenna (Rx) by two OPMs again. Exemplarily, a measured reference signal through air and one measured signal through an object is shown.}
    \label{fig_ImagingSetup}
\end{figure}

Both the emitting and the receiving PCA act as a point-like emitter/detector pair. The transmitting PCA is capable of generating THz pulses with a duration of $\thickapprox$ 5 ps, while the receiving PCA is capable of sampling the electric field of the THz pulse over time. In Figure~\ref{fig_ImagingSetup}, an exemplary reference signal and a signal measured through an object is shown. Two effects, namely absorption and a time delay with respect to the reference signal, are clearly visible and their explanation will be subject to Section~\ref{subsect_ImagingModel}. 

The direct accessibility of amplitude and phase information of the transmitted THz pulses opens up new possibilities for extracting information from the measured signals and modelling the imaging process. Therefore, a customized model for imaging with THz radiation is discussed below.

\subsection{Imaging Model}\label{subsect_ImagingModel}

In general, the propagation of THz radiation is described by the well-known Maxwells equations \cite{Griffiths_Electrodynamics_2014}. For our work on THz imaging, we have chosen to use a more suitable model for our use case based on the approximation of geometrical optics. Geometrical optics can be understood as a limiting case of Maxwells equations with infinitely large radiation frequency. In this common approximation, radiation is generally modelled in the form of mathematical rays. As a consequence, typical phenomena of electromagnetic wave propagation, such as scattering, are not considered in our model. We further assume that effects caused by refraction are negligible. The assumptions made so far are justified, since our work is specifically dedicated to imaging of plastic profiles with planar surfaces only. Furthermore, the interaction of THz radiation with different media is assumed to be linear, isotropic and frequency independent. While assuming linearity and isotropy of the frequently used plastic materials is common practice, we found that in our experimental setup the frequency dependence of material parameters can be sufficiently approximated by a mean value over a given spectral range. Therefore, the remaining two possible ways for radiation to interact with media are absorption and a change in the propagation speed of light, both described by the scalar refractive index $n$ and absorption coefficient $\alpha$ of the material. Absorption is conveniently described by Lambert-Beers law, while the change in the speed of light simply induces a timeshift in the measured THz signal. 

In order to increase the practical applicability of our imaging model, we include the fact that THz radiation is only reasonably focusable up to focal diameters in the mm-range. This practical limitation is considered by assuming that one focussed THz beam consists of multiple parallel geometrical rays, travelling independently. In our model, these independent rays have the same time dependence, but different amplitudes, based on the focal spot intensity distribution (see Figure~\ref{fig_PSF}).

\begin{figure}[ht!]
    \centering
    \includegraphics[viewport=20bp 180bp 550bp 620bp,clip,scale=0.65]{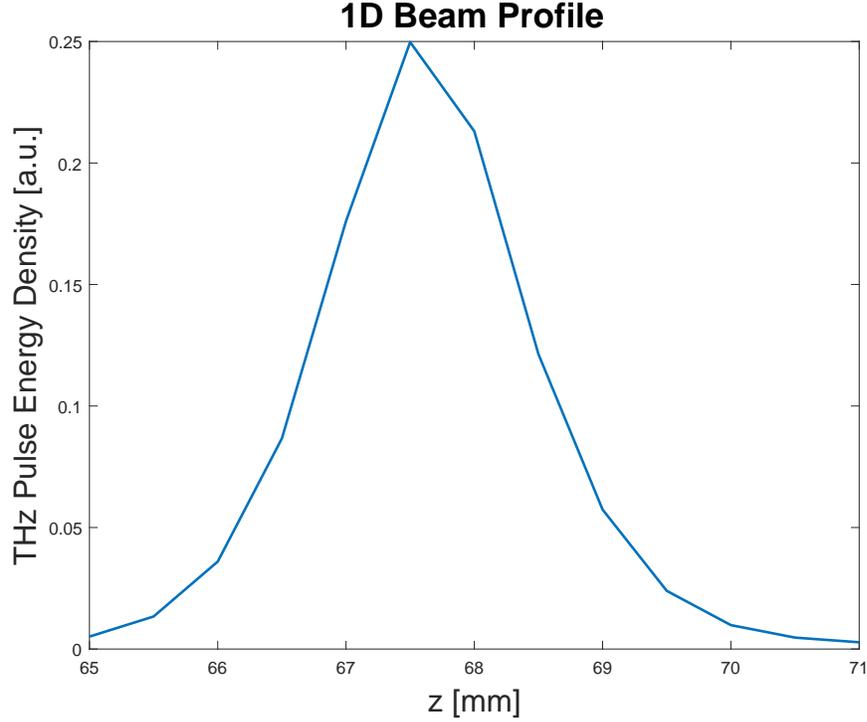}
    \caption{THz beam profile in the focal plane along the parallel scanning direction. The energy density distribution was measured by shifting a rectangular aperture through the focal spot in z-direction.}
    \label{fig_PSF}
\end{figure}

The final assumptions in our proposed imaging model are listed below in five statements, that will act as the basis for futher mathematical derivation of the imaging model in Section~\ref{sect_model}.

\vspace{12pt}
\noindent 
\textbf{Model Assumptions}
\begin{enumerate}
    \item A THz beam behaves like an ensemble of independently travelling parallel rays. The rays have the same time dependence, but different amplitudes based on the focal spot intensity profile.
    \item Each parallel ray experiences damping and a timeshift due to the objects absorption coefficient $\alpha$ and refractive index $n$.  
    \item The timeshift depends on the refractive index $n$ and is directly proportional to the thickness of the penetrated medium.
    \item The transmitted pulse amplitude decreases exponentially with the thickness of the penetrated medium and $\alpha$  (Lambert-Beers law).
    \item Scattering and refraction effects inside the object are negligible.
\end{enumerate}

\section{Mathematical modelling of THz tomography}\label{sect_model}

In this section, we derive mathematical models of THz tomography for the specific setting and physical assumptions described above. These are then used as the starting point for associated reconstruction processes presented in Section~\ref{sect_reconstruction} below. As we shall see, the derived  models are closely related to the Radon transform \cite{Natterer_2001,Louis_1989}, which we thus now briefly review in the following section.

\subsection{The Radon transform}

The Radon transform \cite{Natterer_2001,Louis_1989} maps a density (absorption) function $f(x,y)$ to its line integrals, i.e.,
    \begin{equation}\label{def_Radon}
        (R f)(s,\theta) = \int_\R f(s \vu(\theta) + \sigma  \vu(\theta)^\perp) \, d \sigma \,,
    \end{equation}
where $\vu(\theta) = (\cos(\theta),\sin(\theta))^T$ for different angles $\theta \in [0,2\pi)$. In X-ray tomography, the Radon transform provides a connection between the loss of intensity of an X-ray passing through an object and its density function $f$. In particular, if $I_0(s,\theta)$ and $I(s,\theta)$ denote the initial- and the measured intensity of the X-ray corresponding to the line $L(s,\theta) = L(s\vu(\theta) + \sigma \vu(\theta)^\perp)$, respectively, then
    \begin{equation*}
        I(s,\theta) = I_0(s,\theta) \exp\kl{-(Rf)(s,\theta)} \,.
    \end{equation*}
After rearranging the terms, we obtain the linear operator equation 
    \begin{equation}\label{eq_Radon_Xray}
        (Rf)(s,\theta) = - \log\kl{I(s,\theta)/I_0(s,\theta)} \,,
    \end{equation}
which typically serves as the basis of numerical reconstruction approaches.

\subsection{Nonlinear THz tomography model}

In THz tomography, we do not measure intensities $I(s,\theta)$ but time-dependent electric fields $E(t; s, \theta)$. Furthermore, we do not measure this field for all $(s,\theta)$ but only for certain pairs $(s_i,\theta_j)$. Hence, in the future we use subscripts $i,j$ to emphasize the dependence of quantities on $(s_i,\theta_j)$, e.g., we write $E_\ij(t)$  instead of $E(t; \sitj)$ etc. 

Now, as we saw in the previous section, the fields $E_\ij$ are the result of fields $\Ein_\ij(t,s)$ sent through the object, and of a subsequent focusing of the resulting fields $\Eout_\ij(t,s)$ on a single-pixel detector. Mathematically, this can be described by 
    \begin{equation*}
        E_\ij(t) = \int_{\R} \Eout_\ij(t,s) \, ds \,. 
    \end{equation*}
Next, we need to relate the fields $\Eout_\ij$ to the input fields $\Ein_\ij$. Due to our above assumptions, we can interpret these fields as bundles of rays, each of which experiences both a delay $\DT_\ij(f,s)$ and a damping $D_\ij(f,s)$ caused by the medium, which results in
    \begin{equation}\label{eq_E_Ein_Delay_Damping}
        \Eout_\ij(t,s) = \Ein_\ij(t - \DT_\ij(f,s), s)\times D_\ij(f,s) \,.
    \end{equation}
Since we assumed that the loss of energy is directly proportional to the density $f$ of the medium, we can argue in the same way as for for X-ray tomography \cite{Natterer_2001,Louis_1989} to obtain
    \begin{equation*}
        D_\ij(f,s) =  \exp\kl{-\frac{1}{2}(Rf)(s,\theta_j)}\,,
    \end{equation*}   
where the factor $1/2$ appears for scaling purposes. Together with \eqref{eq_E_Ein_Delay_Damping} this yields
    \begin{equation}\label{eq_E_Ein_Delay}
        E_\ij(t) = \int_{\R} \Ein_\ij(t-\DT_\ij(f,s),s) \exp\kl{-\frac{1}{2}(Rf)(s,\theta_j)} \, ds \,.
    \end{equation}
Furthermore, due to the experimental setup (see Figure~\ref{fig_PSF} and \ref{fig_ImagingSetup}), there holds
    \begin{equation*}
        \Ein_\ij(t,s) = w(s-s_i) \Eref(t) \,,
    \end{equation*}
where $\Eref(t)$ is a reference field and $w(s)$ is a weight function, both of which can be determined experimentally. Combining this with \eqref{eq_E_Ein_Delay} we obtain
    \begin{equation}\label{eq_E_Eref_Delay}
        E_\ij(t) = \int_{\R} w(s-s_i) \Eref(t-\DT_\ij (f,s))  \exp\kl{-\frac{1}{2}(Rf)(s,\theta_j)} \, ds \,.
    \end{equation}
Now, it is also possible to derive explicit expressions for the delay (see below). However, we can also get rid of it by integrating the above equation with respect to time, which yields
    \begin{equation}\label{eq_E_int}
    \begin{split}
        \int_{-\infty}^{\infty} E_\ij(t) \, dt = \int_{-\infty}^{\infty} \Eref(t) \, dt \int_{\R} w(s-s_i)  \exp\kl{-\frac{1}{2}(Rf)(s,\theta_j)} \, ds  \,.
    \end{split}
    \end{equation}
Hence, if we define the quantities
    \begin{equation*}
        P_\ij := \int_{-\infty}^{\infty} E_\ij(t) \, dt \,,
        \qquad
        \text{and}
        \qquad
        \Pref := \int_{-\infty}^{\infty} \Eref(t) \, dt \,,
    \end{equation*}
then \eqref{eq_E_int} can be written in the form
    \begin{equation}\label{eq_f_nonlin}
        \Pref \int_{\R} w(s-s_i)  \exp\kl{-\frac{1}{2}(Rf)(s,\theta_j)} \, ds  = P_\ij  \,,
    \end{equation}
and thus we arrive at the following

\begin{problem}\label{prob_F}
Given the data $P_\ij$ and $\Pref$, as well as the weight function $w(s)$, the full-beam THz tomography problem consists in finding the density function $f$ as the solution of the nonlinear system of equations \eqref{eq_f_nonlin}. 
\end{problem}

\subsection{Linear THz tomography models}

In practical applications, the weight function $w(s)$ often takes the shape of a Gaussian function centered at $0$, which becomes narrower the shorter the wavelength and the larger the numerical aperture. Hence, it is not unreasonable to formally approximate $w(s)$ by the delta distribution, which corresponds to the case that the field $\Ein_\ij(t,s)$ consists only of a single ray (at $s=0$). With this, equation \eqref{eq_E_Eref_Delay} simplifies to 
    \begin{equation}\label{eq_E_single}
        E_\ij(t) =  \Eref(t-\DT_\ij(f,s_i))  \exp\kl{-\frac{1}{2}(Rf)(s_i,\theta_j)}  \,.
    \end{equation}   
Now we can proceed in two ways. On one hand, we can again integrate the equation with respect to time, which yields 
    \begin{equation}\label{eq_Radon_exp_P}
        P_\ij =  \Pref  \exp\kl{-\frac{1}{2}(Rf)(s_i,\theta_j)}  \,.
    \end{equation}  
which after rearranging becomes
    \begin{equation}\label{eq_Radon_P}
        (Rf)(s_i,\theta_j) = -2 \log\kl{P_\ij/\Pref} \,.
    \end{equation}
This leads us to 

\begin{problem}\label{prob_R_P}
Given the data $P_\ij$ and $\Pref$, the single-ray THz tomography problem consists in finding the density function $f$ by solving the system of equations \eqref{eq_Radon_P}. 
\end{problem}

Alternatively, we can first square equation \eqref{eq_E_single} before integrating it, which yields
    \begin{equation*}
        \int_{-\infty}^{\infty} E_\ij(t)^2 \, dt =  \exp\kl{-(Rf)(s_i,\theta_j)} \int_{-\infty}^{\infty} \Eref(t)^2 \, dt   \,.
    \end{equation*} 
Now if we define the quantities
    \begin{equation*}
        I_\ij := \int_{-\infty}^{\infty} E_\ij(t)^2 \, d t \,,
        \qquad 
        \text{and}
        \qquad
        \Iref := \int_{-\infty}^{\infty} \Eref(t)^2 \, d t \,,
    \end{equation*}
which are exactly the intensities of the fields $E_\ij$ and $\Eref$, respectively, we obtain
    \begin{equation*}
        I_\ij =  \exp\kl{-(Rf)(s_i,\theta_j)} \Iref \,.
    \end{equation*}   
Rearranging this equation we get
    \begin{equation}\label{eq_Radon_I}
        (Rf)(s_i,\theta_j) = - \log\kl{I_\ij/\Iref} \,,
    \end{equation}
which is exactly \eqref{eq_Radon_Xray} with $s=0$. Hence, our THz tomography model can be seen as a specific generalization of the X-ray Radon transform model. This now leads us to 

\begin{problem}\label{prob_R_I}
Given the data $I_\ij$ and $\Iref$, the single-ray intensity THz tomography problem consists in finding the density function $f$ by solving the system of equations \eqref{eq_Radon_I}. 
\end{problem}

Note that while Problem~\ref{prob_F} is a nonlinear problem, both Problem~\ref{prob_R_P} and Problem~\ref{prob_R_I} are linear. Each of is used as basis for an associated reconstruction process presented in Section~\ref{sect_reconstruction} below. 

\subsection{Nonlinear model for uniform samples}
  
In case that the scanned sample only consists of a single material and air, the density function $f$ only takes two values. These are $0$ for air and $\alpha_M \neq 0$ for the material. In this case the delay $\DT_\ij(f,s)$ can be computed explicitly via the kinematic equation 
    \begin{equation*}
        \text{time} = \text{distance}/\text{velocity} 
    \end{equation*}
as follows: Consider a ray passing through the object along a line corresponding to $(s,\theta_j)$. The total distance which this ray travels inside material is then given by
    \begin{equation*}
        \text{distance} =  \int_\R \chisupf(s\vu(\theta_j) + \sigma \vu(\theta_j)^\perp) \, d \sigma \,,
    \end{equation*} 
where $\chisupf$ denotes the indicator function of the support of $f$. Denoting by $c_0$ the speed of light in vacuum and by $n$ the refractive index of the material, it follows that the total time which the ray travels inside material is given by
    \begin{equation*}
        \text{time} = \text{distance}/\text{velocity} 
        =
        \frac{n}{c_0}\int_\R \chisupf(s\vu(\theta_j) + \sigma \vu(\theta_j)^\perp) \, d\sigma \,.
    \end{equation*} 
If instead the ray had travelled the same distance through air, the time which this would have taken can be calculated the same way but with $n$ replaced by the refractive index of air $n_0$. Thus, the delay $\DT_\ij(f,s)$ due to the ray traveling through material is
    \begin{equation*}
        \DT_\ij(f,s) = \frac{n - n_0}{c_0} \int_\R \chisupf(s\vu(\theta_j) + \sigma \vu(\theta_j)^\perp) \, d\sigma =   \frac{n - n_0}{\alpha_M c_0} (Rf)(s,\theta_j) \,. 
    \end{equation*}
Now, combining this with \eqref{eq_E_Ein_Delay} we obtain
    \begin{equation}\label{eq_E_Eref_special}
        E_\ij(t) = \int_{\R} w(s-s_i) \Eref\kl{t-\tfrac{n - n_0}{\alpha_M c_0} (Rf)(s,\theta_j)}  \exp\kl{-\frac{1}{2}(Rf)(s,\theta_j)} \, ds \,,
    \end{equation}
which leads us to the following 

\begin{problem}\label{prob_F_E}
Given the data $E_\ij(t)$ and $\Eref(t)$, the time-dependent full-beam THz tomography problem for uniform material samples consists in finding the density function $f$ which satisfies \eqref{eq_E_Eref_special}. 
\end{problem}

Note that the structure of the sample, more precisely the fact that it is composed from a single material, directly enters into Problem~\ref{prob_F_E}. For samples composed of multiple different materials, similar formulas can be derived, which however lead to much more complicated mathematical models, which are out of scope of the present paper.

\section{Reconstruction approaches}\label{sect_reconstruction}

In this section, we consider various reconstruction approaches for obtaining solutions to the Problems~\ref{prob_F}-\ref{prob_R_I} introduced above. For this, we differentiate between the nonlinear Problem~\ref{prob_F} and the linear Problems~\ref{prob_R_P} and \ref{prob_R_I}, which are based on the same operator and can thus be treated in a similar way. Since each of those problems is ill-posed, regularization plays a crucial role in all of the following considerations \cite{Engl_Hanke_Neubauer_1996}.

\subsection{Reconstruction approach for the nonlinear model}\label{subsect_NonlinearReconstruction}

In this section, we derive a reconstruction approach for solving Problem~\ref{prob_F} based on nonlinear Landweber iteration. We start by defining the nonlinear operator
    \begin{equation*}
        G(f)(s,\theta) := \int_{\R} w(r-s)  \exp\kl{-\frac{1}{2}(Rf)(r,\theta)} \, d r \,.
    \end{equation*}
With this, equation \eqref{eq_f_nonlin} can be written in the form
    \begin{equation}\label{Gfi=g}
        G(f)(\sitj) = P_\ij/\Pref \,,
    \end{equation}
which is a discrete version of the nonlinear operator equation
    \begin{equation}\label{Gf=g}
        G(f) = g \,,
    \end{equation}
where the right-hand side $g$ is defined by
    \begin{equation*}
        g(s,\theta) := \kl{\int_{-\infty}^{\infty} E(t; s,\theta) \, dt }\Big/\kl{ \int_{-\infty}^{\infty}\Eref(t) \, dt }\,.
    \end{equation*}
Equation \eqref{Gf=g} models the situation that measurements of the electric fields $E$ and $\Eref$ are available for all parallel lines $s$ and angles $\theta$. Thus, a reconstruction approach for Problem~\ref{prob_F} can be obtained by discretizing a solution approach for the continuous equation \eqref{Gf=g}. Hence, we now turn our attention to \eqref{Gf=g}. 

First of all, we need to fix suitable definition and image spaces for $G$. Without loss of generality we can assume that the density function $f$ has compact support within the domain $\OD := \{x \in \R^2 \, \vert \, \abs{x} \leq 1\}$. Now if in addition we define the domain $\OS := [-1,1]\times[0,2\pi)$, then it is known \cite{Natterer_2001,Louis_1989} that
    \begin{equation*}
        R \, : \, \LtOD \to \LtOS \,,
    \end{equation*}
is a bounded linear operator. Furthermore, for any $f$ with compact support within $\OD$ there holds $(Rf)(s,\theta) = 0$ for all $\abs{s} > 1$. Hence, its natural extension 
    \begin{equation*}
        R \, : \, \LtOD \to \Lt(\R\times[0,2\pi)) \,,
    \end{equation*}
is a bounded linear operator as well. Unfortunately, these results do not hold for the operator $G$, due to the presence of the exponential function. A similar problem also has been encountered with the attenuated Radon transform \cite{Natterer_2001,Louis_1989}, which in dependence of the so-called attenuation function $\mu$ is defined by
    \begin{equation*}
        (R_\mu f)(s,\theta) = \int_\R f(s\vu(\theta) + \sigma \vu(\theta)^\perp) \exp\kl{-\int_{\sigma}^\infty \mu(s\vu(\theta) + \rho \, \vu(\theta)^\perp) \, d\rho} \, d\sigma \,.
    \end{equation*}
A popular remedy proposed in \cite{Dicken_1999,Dicken_1998} is based on the fact that for realistic density functions $f$ there holds $f \geq 0$ and consequently also $R f \geq 0$. Hence, the exponential function can be replaced by a function $\E \in C^2(\R)$, which is such that
    \begin{equation*}
        \E(x) = \exp(-x) \,, \qquad \forall \, x \in \R_0^+ \,,
    \end{equation*}
and for which $\abs{\E}$, $\abs{\E'}$, and $\abs{\E''}$ are bounded. With this, we can define the operator
    \begin{equation}\label{def_F}
    \begin{split}
        F \, : \, \LtOD &\to \LtOS
        \\
        f &\mapsto F(f)(s,\theta) := \int_{\R} w(r-s)  \E\kl{\foh(Rf)(r,\theta)} \, dr \,,
    \end{split}
    \end{equation}
which coincides with $G$ for any $f \geq 0$, and instead of \eqref{Gf=g} consider the equation
    \begin{equation}\label{Ff=g}
        F(f) = g \,.
    \end{equation}
For the operator $F$ it is now possible to prove the following well-definedness result:
\begin{proposition}
For any $w \in \LoR$ the operator $F$ given in \eqref{def_F} is well-defined. 
\end{proposition}
\begin{proof}
Since from the definition of $F$ there follows
    \begin{equation*}
    \begin{split}
        \norm{F(f)}_\LtOS^2 = \int_{-1}^1 \int_0^{2\pi}  \kl{\int_{\R} w(r-s)  \E\kl{\foh(Rf)(r,\theta)} \, dr}^2 \, d\theta \, ds 
       \\
       \leq 4 \pi \norm{w}_\LoR^2   \norm{\E}_{\Linf(\R)}^2  \,,
    \end{split}
    \end{equation*}
the statement follows from the assumption that $w \in \LoR$ and that $\abs{\E}$ is bounded. 
\end{proof}

For our reconstruction approach, we need the following
\begin{proposition}
For any $w \in \LoR \cap \LtR$ the operator $F: \LtOD \to \LtOS$ defined in \eqref{def_F} is continuously Fr\'echet differentiable with
    \begin{equation}\label{def_DF}
        (F'(f)h)(s,\theta) :=  \foh \int_{\R} w(r-s)  \E'\kl{\foh(Rf)(r,\theta)} (Rh)(r,\theta)    \, dr \,.
    \end{equation} 
\end{proposition}
\begin{proof}
For any $f \in \LtOD$ we define the linear operator
    \begin{equation*}
        (A(f)h)(s,\theta) := \foh \int_{\R} w(r-s)  \E'\kl{\foh(Rf)(r,\theta)} (Rh)(r,\theta)    \, dr \,.
    \end{equation*}
Now since for any $h \in \LtOD$ there holds
    \begin{equation*}
    \begin{split}
        & \norm{A(f)h}_\LtOS^2 = \int_{-1}^1 \int_0^{2\pi}  \kl{\foh \int_{\R} w(r-s)  \E'\kl{\foh(Rf)(r,\theta)} (Rh)(r,\theta)\, dr}^2 \, d\theta \, ds 
       \\
       & \qquad \leq
       \frac{1}{4} \norm{\E'}_{\LiR}^2 \int_{-1}^1 \int_0^{2\pi}  \int_{\R} \abs{w(r-s) }^2 \,d r  \int_{\R} \abs{(Rh)(r,\theta)}^2 \, dr \, d\theta \, ds 
       \\
       &\qquad =
       \foh \norm{w}_\LtR^2 \norm{\E'}_{\LiR}^2   \norm{Rh}_{\LtOS}^2 \,,
    \end{split}
    \end{equation*}
the boundedness of $R$ and $\abs{\E'}$ together with $w \in \LtR$ implies that
    \begin{equation*}
        \norm{A(f)h}_\LtOS^2 \leq
       \foh \norm{w}_\LtR^2 \norm{\E'}_{\LiR}^2   \norm{Rh}_{\LtOS}^2
       \leq C \norm{h}_\LtOD^2 \,,
    \end{equation*}
and thus $A(f)h : \LtOD \to \LtOS$ is a bounded linear operator. Hence, it remains to show that $A(f)h$ is in fact the Fr\'echet derivative of $F$. For this, we start by looking at
    \begin{equation}\label{eq_helper_1}
    \begin{split}
        &\kl{F(f+h)-F(f) -A(f)h}(s,\theta)
        =  \int_{\R} w(r-s)  \E\kl{\foh(Rf)(r,\theta) + \foh(Rh)(r,\theta)} \, dr 
        \\
        & \quad - \int_{\R} w(r-s) \E\kl{\foh(Rf)(r,\theta)} \, dr - \int_{\R} w(r-s)  \foh \E'\kl{\foh(Rf(r,\theta))} (Rh)(r,\theta) \, dr
    \end{split}
    \end{equation}
Now since $\E$ is twice continuously differentiable, for any $a,b \in \R$ there holds
    \begin{equation*}
        \abs{\E(a+b) - \E(a) - \E'(a)b} = \abs{\foh \int_a^{a+b}(a+b-t) \E''(t)  \, dt } 
        \leq \frac{1}{4} \norm{\E''}_\LiR \abs{b}^2 \,.
    \end{equation*} 
Using this together with the choice
    \begin{equation*}
        a = \foh(Rf)(r,\theta)\,,
        \qquad \text{and} \qquad
        b = \foh(Rh)(r,\theta)\,,
    \end{equation*}
it thus follows from \eqref{eq_helper_1} that
    \begin{equation*}
    \begin{split}
        \abs{F(f+h)-F(f) -A(f)h }(s,\theta) 
        &\leq  \frac{1}{4} \norm{\E''}_\LiR \int_{\R} \abs{w(r-s) } \abs{\foh Rh(r,\theta)}^2 \, dr \,,
    \end{split}
    \end{equation*}
which implies
    \begin{equation*}
    \begin{split}
        &\norm{F(f+h)-F(f) -A(f)h}_\LtOS^2 
        =
        \int_{-1}^1 \int_0^{2\pi}  \abs{F(f+h)-F(f) -A(f)h }^2 \, d\theta \, ds 
        \\
        & \qquad \leq 2^{-8}\norm{\E''}_\LiR^2
        \int_{-1}^1 \int_0^{2\pi} \kl{  \int_{\R} \abs{w(r-s) }  \abs{Rh(r,\theta)}^2 \, dr }^2 \, d\theta \, ds \,.
    \end{split}
    \end{equation*}
Since by the Cauchy-Schwarz inequality we have
    \begin{equation*}
    \begin{split}
        & \int_{-1}^1 \int_0^{2\pi} \kl{  \int_{\R} \abs{w(r-s) }  \abs{Rh(r,\theta)}^2 \, dr }^2 \, d\theta \, ds 
        \leq 
        2 \norm{w}_\LtR^2  \int_0^{2\pi} \int_{\R}  \abs{Rh(r,\theta)}^4 \, dr  \, d\theta \,,
    \end{split}
    \end{equation*}
and since there holds
    \begin{equation*}
        \int_0^{2\pi} \int_{\R}  \abs{Rh(r,\theta)}^4 \, dr  \, d\theta
        =
        \int_0^{2\pi} \int_{-1}^1  \abs{Rh(r,\theta)}^4 \, dr \, d\theta
        = \norm{Rh}_{L_4(\OS)}^4 \,,
    \end{equation*}
it follows that
    \begin{equation*}
        \norm{F(f+h)-F(f) -A(f)h}_\LtOS^2 
        \leq 
        2^{-7}\norm{w}_\LtR^2 \norm{\E''}_\LiR^2  \norm{Rh}_{L_4(\OS)}^4
    \end{equation*}
Combining this with the fact that 
    \begin{equation*}
        \norm{Rh}_{L_4(\OS)} \leq \abs{\OS}^{-\frac{1}{4}} \norm{Rh}_\LtOS
        \leq 
        \kl{4\pi}^{-\frac{1}{4}} \norm{R} \norm{h}_\LtOD \,,
    \end{equation*}
we obtain 
    \begin{equation*}
        \norm{F(f+h)-F(f) -A(f)h}_\LtOS
        \leq 
        \kl{2^9 \pi}^{-1/2} \norm{w}_\LtR \norm{\E''}_\LiR  \norm{R}^2 \norm{h}_\LtOD^2 \,.
    \end{equation*}
Together with the boundedness of $R$ and $\abs{\E''}$, and since $w\in\LtR$, this implies that $A(f)h$ is the Fr\'echet derivative of $F$, which yields the assertion. 
\end{proof}

Next, we characterize the adjoint of the Fr\'echet derivative of $F$ in the following
\begin{proposition}
Let $w \in \LoR \cap \LtR$ and let $F : \LtOD \to \LtOS$ be defined as in \eqref{def_F}. Then
    \begin{equation}
        (F'(f)^*g)(x,y) :=   R^*\kl{ \foh \E'\kl{\foh(Rf)(s,\theta)}  \int_{-1}^1  w(s-r)  g(r,\theta)\, dr }(x,y)   \,,
    \end{equation} 
where the adjoint of the Radon transform $R^*$ is understood w.r.t.\ the variables $s$ and $\theta$.
\end{proposition}
\begin{proof}
For any $g \in \LtOS$ it follows from \eqref{def_DF} that
    \begin{equation*}
    \begin{split}
        &\spr{F'(f)h,g}_\LtOS 
        =
        \int_{-1}^1 \int_0^{2\pi} (F'(f)h)(s,\theta) g(s,\theta) \, d\theta \, ds
        \\
        & \qquad =
        \int_{-1}^1 \int_0^{2\pi} \kl{\foh \int_{\R} w(r-s)  \E'\kl{\foh(Rf)(r,\theta)} (Rh)(r,\theta)    \, dr  }g(s,\theta) \, d\theta \, ds
        \\
        & \qquad =
        \int_{\R} \int_0^{2\pi}(Rh)(r,\theta) \kl{ \foh \E'\kl{\foh(Rf)(r,\theta)}  \int_{-1}^1 w(r-s)  g(s,\theta)\, ds} \, d\theta  \, dr \,.
    \end{split}
    \end{equation*}
Thus, defining the operator
    \begin{equation*}
        (B(f)g)(s,\theta) :=  \foh \E'\kl{\foh(Rf)(s,\theta)} \int_{-1}^1  w(s-r)  g(r,\theta)\, dr \,,
    \end{equation*}
we obtain
    \begin{equation*}
    \begin{split}
        \spr{F'(f)h,g}_\LtOS 
        &=
        \int_{\R} \int_0^{2\pi}(Rh)(s,\theta) (B(f)g)(s,\theta) \, d\theta  \, ds 
        \\
        & =
        \int_{-1}^1 \int_0^{2\pi}(Rh)(s,\theta) (B(f)g)(s,\theta) \, d\theta  \, ds 
        \\
        & = 
        \spr{Rh,B(f)g}_\LtOS = \spr{h,R^*B(f)g}_\LtOD
        \,,
    \end{split}
    \end{equation*}
which yields the assertion.
\end{proof}

Having derived the Fr\'echet derivative and its adjoint, we are now in the position to consider a solution approach for \eqref{Ff=g}. A popular iterative reconstruction approach is nonlinear Landweber iteration \cite{Engl_Hanke_Neubauer_1996,Kaltenbacher_Neubauer_Scherzer_2008}, defined by
    \begin{equation}\label{Landweber_nonlin}
        f_{k+1}^\delta = f_k^\delta + \gamma_k^\delta F'(f_k^\delta)^*\kl{g^\delta - F(f_k^\delta)} \,, 
    \end{equation}
where $\gamma_k^\delta$ is stepsize, either chosen as constant or e.g.\ as the steepest descent stepsize \cite{Scherzer_1996}:
    \begin{equation}\label{steepest_descent}
        \gamma_k^\delta := \frac{\norm{s_k^\delta}^2}{\norm{F'(f_k^\delta) s_k^\delta}^2} \,,
        \qquad
        s_k^\delta := F'(f_k^\delta)^*\kl{g^\delta - F(f_k^\delta)} \,.
    \end{equation}
Typically, the iteration is combined with the discrepancy principle \eqref{discrepancy_principle} as a stopping rule, which now determines the stopping index $k_*$ by
    \begin{equation*}
	    k_*:=\inf\Kl{k\in\mathbb{N} \, \vert \, \norm{F(f_k^\delta) - g^\delta} \leq \tau\delta} \,.
	\end{equation*}

After these considerations, we can now define a solution approach for Problem~\ref{prob_F}. Since for all $f\geq 0$ there holds $F(f) = G(f)$, it follows from \eqref{Gfi=g} that
    \begin{equation*}
        F(f)(\sitj) = g(\sitj)   \,.
    \end{equation*}
Hence, for solving Problem~\ref{prob_F} we propose to use the Landweber iteration \eqref{Landweber_nonlin} for the operator $F$ together with a discretization for the space $\LtOS$ based on collocation at the points $(\sitj)$, combined with either a constant or the steepest descent stepsize.

\subsection{Reconstruction approaches for the linear models} \label{subsect_LinearReconstruction}

In this section we consider reconstruction approaches for Problems~\ref{prob_R_P} and \ref{prob_R_I}, which both amount to solving a linear system of the form
    \begin{equation}\label{R_si_theta=g}
        (Rf)(s_i,\theta_j) = g(s_i,\theta_j) \,.
    \end{equation}
Since this is a discrete version of the continuous Radon transform equation 
    \begin{equation}\label{Rf=g}
        (R f)(s,\theta) = g(s,\theta) \,,
    \end{equation}
reconstruction approaches based on methods for solving \eqref{Rf=g} suggest themselves. As always when dealing with inverse problems, special attention needs to be given to the fact that instead of $g$ one typically only has access to noisy data $g^\delta$, which are assumed to satisfy 
    \begin{equation*}
        \norm{g- g^\delta} \leq \delta \,,
    \end{equation*}
where $\delta$ denotes the noise level. Among the earliest approaches for dealing with this issue is the so-called filtered back-projection \cite{Natterer_2001,Louis_1989}. It combines the classic Radon inversion formula with a suitable filter, which acts as a regularization stabilizing the inversion. Another very prominent approach is Tikhonov regularization \cite{Engl_Hanke_Neubauer_1996}, which defines the approximation $f_\beta^\delta$ of the density $f$ as the minimizer of the Tikhonov functional
    \begin{equation*}
        f  \, \mapsto \, \norm{Rf- g^\delta}^2_\Lt + \beta \norm{f}^2_\Lt \,,
    \end{equation*}
where $\beta = \beta(\delta,\yd)$ is a regularization parameter. Many variations of this approach are possible, which allow to include different a-priori information on the density function $f$, for example knowledge on its sparsity with respect to a given basis.

Alternatively to Tikhonov regularisation, perhaps the most well-known iterative regularization method is Landweber iteration \cite{Engl_Hanke_Neubauer_1996}, whose iterates are defined by 
    \begin{equation*}
        f_{k+1}^\delta = f_{k}^\delta + \gamma R^*(g^\delta - R f_k^\delta) \,,
    \end{equation*}
where $\gamma$ is a stepsize parameter. In order to obtain a convergent regularization method, it has to be combined with a suitable stopping rule. The most prominent choice is the discrepancy principle, which determines a stopping index $k_*$ via
    \begin{equation}\label{discrepancy_principle}
	    k_*:=\inf\Kl{k\in\mathbb{N} \, \vert \, \norm{R f_k^\delta - g^\delta} \leq \tau\delta},
	\end{equation}
for some parameter $\tau > 1$. Similarly to Tikhonov regularisation, also Landweber iteration can be adapted in many different ways. Sparsity assumptions can be incorporated by adapting the iteration scheme to
    \begin{equation*}
        f_{k+1}^\delta = S \kl{f_{k}^\delta + \gamma R^*(g^\delta - R f_k^\delta) }\,,
    \end{equation*}
where $S$ is a shrinkage/thresholding operator, which gives rise to ISTA. This can be combined with different stepsizes and acceleration schemes, which for example gives rise to the well-known method known as FISTA.

A conceptually different approach for the solution of \eqref{R_si_theta=g} is contour tomography, which as the name suggests aims not at reconstructing the density function $f$ but its contours. On its most basic level, it amounts to a differentiation of the data, followed by a back-projection, the result of which is a function whose jumps are roughly equivalent to the jumps of the original density function $f$. For details, we refer to \cite{Louis_Maass_1993}.

In Section~\ref{sect_numerical_results}, we apply the reconstruction methods discussed above (filtered back-projection, Tikhonov regularisation, Landweber iteration, contour tomography) to obtain solutions to both Problems~\ref{prob_R_P} and \ref{prob_R_I}.

\section{Numerical results}\label{sect_numerical_results}

In this section, we present a number of numerical results demonstrating the usefulness of our different reconstruction approaches. These tests are based on experimental data obtained from THz measurements of the plastic sample depicted in Figure~\ref{fig_plasticsample} (left).

\begin{figure}[ht!]
    \centering
    \includegraphics[scale=0.41]{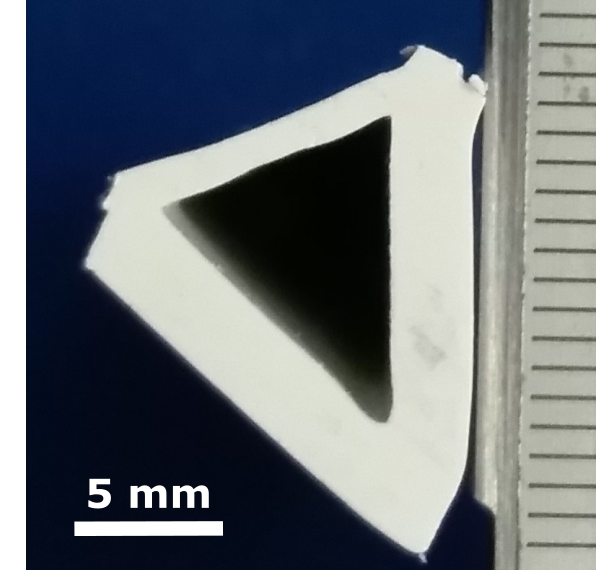}
    \qquad
    \includegraphics[viewport=45bp 210bp 545bp 630bp,clip,scale=0.43]{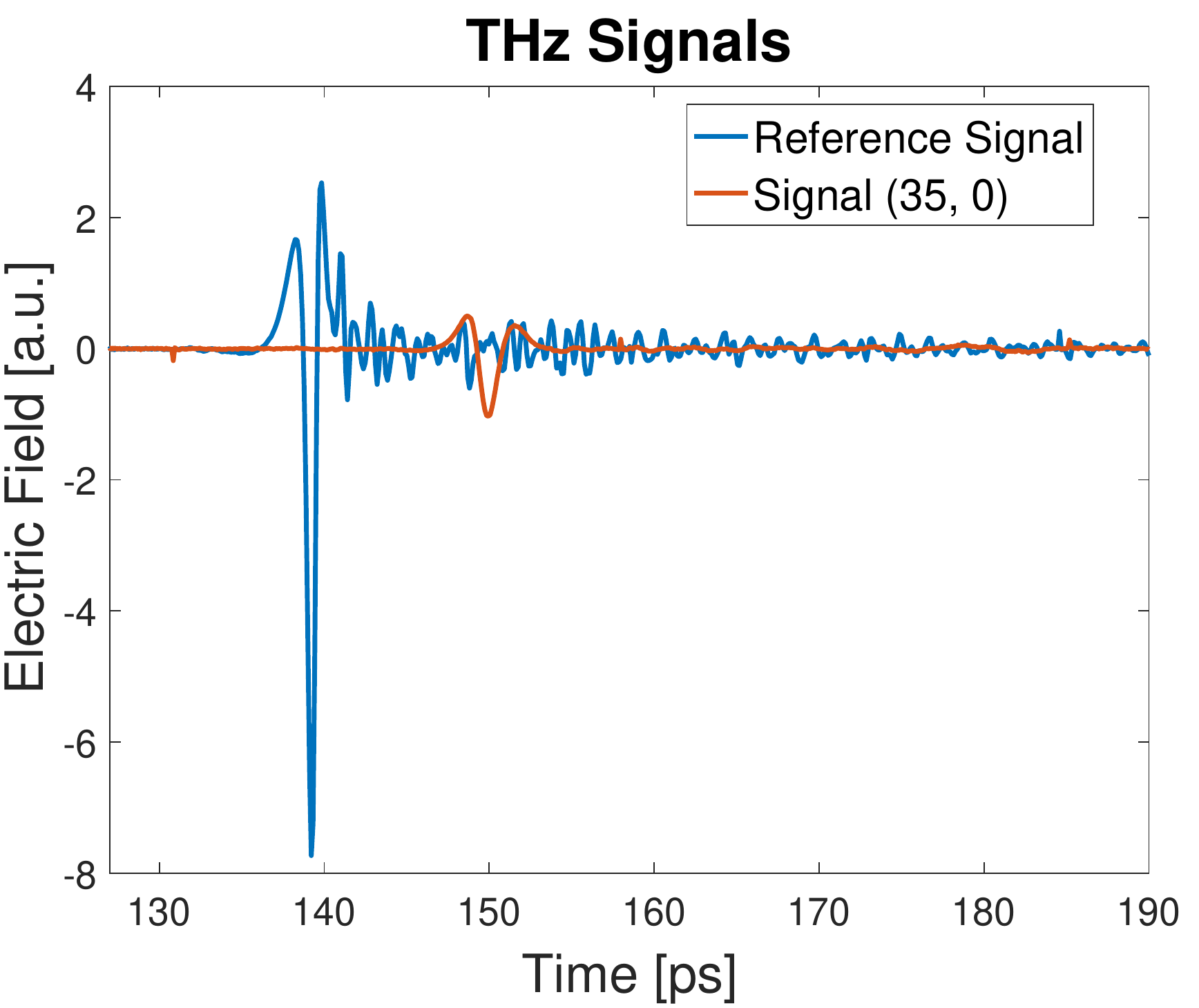}
    \caption{Triangular plastic sample (left), the measured electric field $E_\ij$ corresponding to $(\sitj) = (35,0)$ (right, blue), and the reference field $\Eref$ (right, orange).}
    \label{fig_plasticsample}
\end{figure}

The sample was scanned from $360$ uniformly distributed angles $\theta_j$ using $71$ equally spaced parallel beams $s_j$. An example of the measured electric field $E_\ij$ corresponding to $(\sitj) = (35,0)$ as well as the reference field $\Eref$ is depicted in Figure~\ref{fig_plasticsample} (right). Note that the reference field $\Eref$ was determined by averaging over consecutive measurements of a THz beam travelling through air only. 

As can already be seen from these examples, both $E_\ij$ and $\Eref$ contain a certain amount of measurement noise, which mostly affects the signal quality away from the main peak, see Figure~\ref{fig_mainpeak_extracted}. While this does not have a strong influence on the reconstructions obtained from intensity data $I_\ij$ and $\Iref$ (i.e.\ in Problem~\ref{prob_R_I}), it does affect the reconstruction when $P_\ij$ and $\Pref$ are used as data (i.e.\ in Problem~\ref{prob_F} and Problem~\ref{prob_R_P}). Fortunately, the most relevant information of the data is contained in the main peak, even though different sections of the beam might go trough different thicknesses of the material. Hence, we pre-processed the measured electric fields $E_\ij$ and $\Eref$ by extracting and considering only the main peak of each signal (see Figure~\ref{fig_mainpeak_extracted} for an example). 

\begin{figure}[ht!]
    \centering
    \includegraphics[viewport=170bp 480bp 310bp 560bp,clip, scale=1.5]{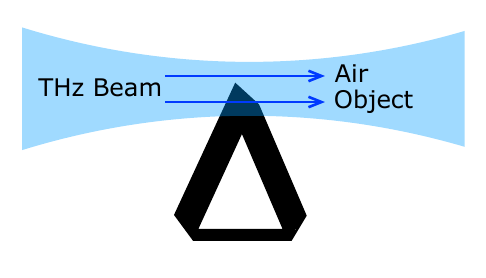}
    \includegraphics[viewport=30bp 210bp 540bp 630bp,clip, scale=0.4]{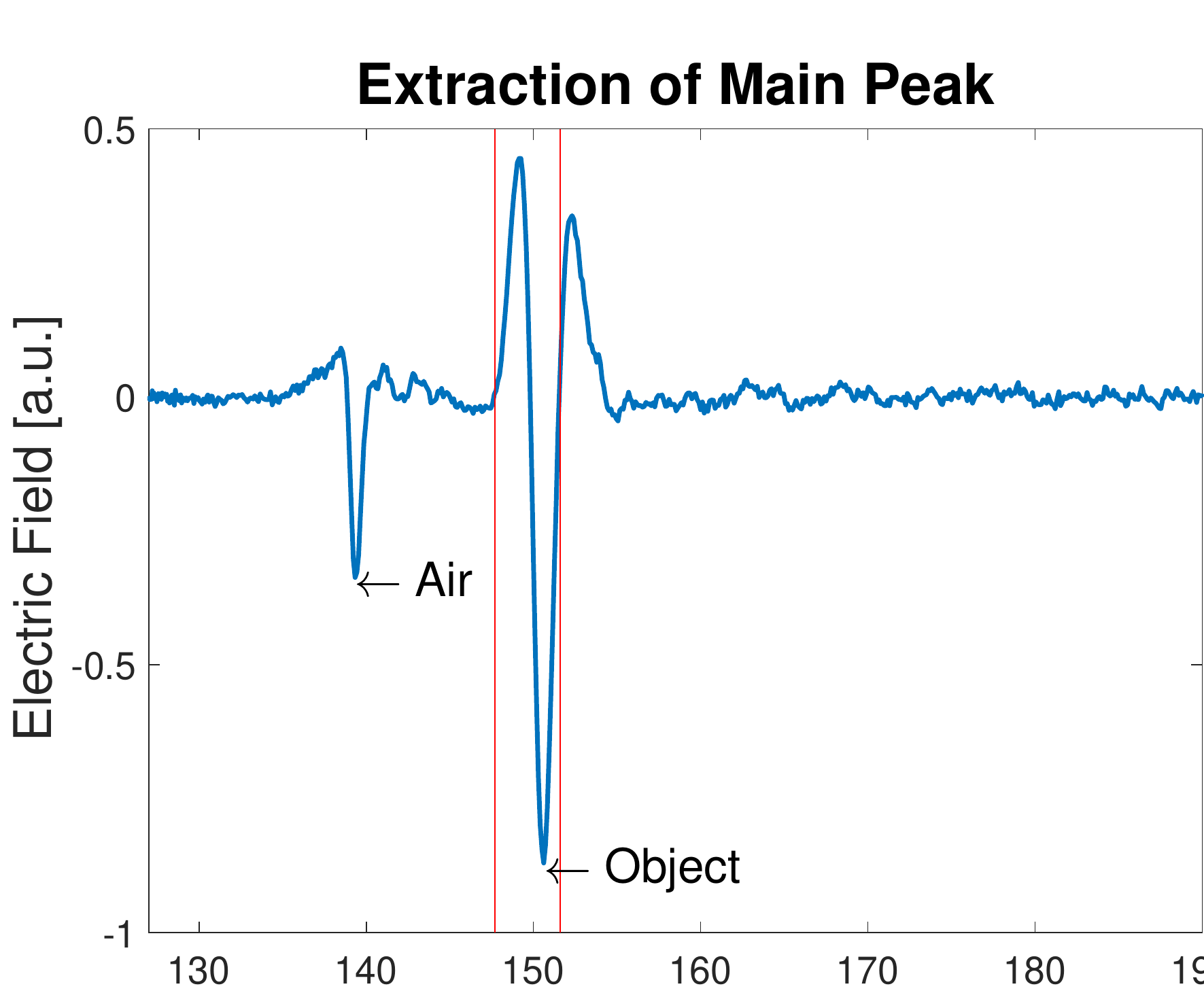}
    \caption{An example for the presence of multiple peaks in a THz signal. (left) The THz beam partially travels through air and the object. (right) This gives rise to two dominant peaks in the THz signal, one arising from the pulse that travelled through air and a second one that travelled through the object respectively. For our reconstructions only the largest peak (main peak) is used.}
    \label{fig_mainpeak_extracted}
\end{figure}

Another technicality when dealing with $P_\ij$ and $\Pref$ as data has to do with the fact that in theory, all $P_\ij$ should have the same sign as $\Pref$; compare to \eqref{eq_f_nonlin} and \eqref{eq_Radon_exp_P}. Unfortunately, this is not satisfied for measured $P_\ij$ due to noise and discretization  errors in the signal. Apart from the analytical inconsistencies which it introduces, this also leads to numerical problems when one has to compute the right-hand side $-2\log\kl{P_\ij/\Pref}$ of \eqref{eq_Radon_P} for solving Problem~\ref{prob_R_P}. However, note that mathematically both \eqref{eq_f_nonlin} and \eqref{eq_Radon_exp_P} remain the same if one first divides by $\Pref$ and then applies the absolute value on both sides. Hence, for obtaining the results presented below, we always used $\abs{P_\ij/\Pref}$ instead of $P_\ij/\Pref$ in order to circumvent these issues. 

When working with experimental data, it is advantageous to carry out a number of pre-processing steps on the quantities $\abs{P_\ij/\Pref}$ and $I_\ij/\Iref$ before computing the actual data $-2\log(\abs{P_\ij/\Pref})$ and $-\log(I_\ij/\Iref)$ for Problem~\ref{prob_R_P} and Problem~\ref{prob_R_I}, respectively. In our case, this pre-processing consists of a thresholding to remove unnaturally large values induced by measurement errors, a suitable scaling, and the application of a Gaussian filter to remove some high-frequency noise components. The resulting data (sinograms) are depicted in Figure~\ref{fig_sinogram_P_I}. The data $\abs{P_\ij/\Pref}$ corresponding to the nonlinear Problem~\ref{prob_F} is pre-processed in the same way and is depicted in Figure~\ref{fig_nonlinear_Landweber_exp} (left).

\begin{figure}[ht!]
    \centering
    \includegraphics[scale=0.5]{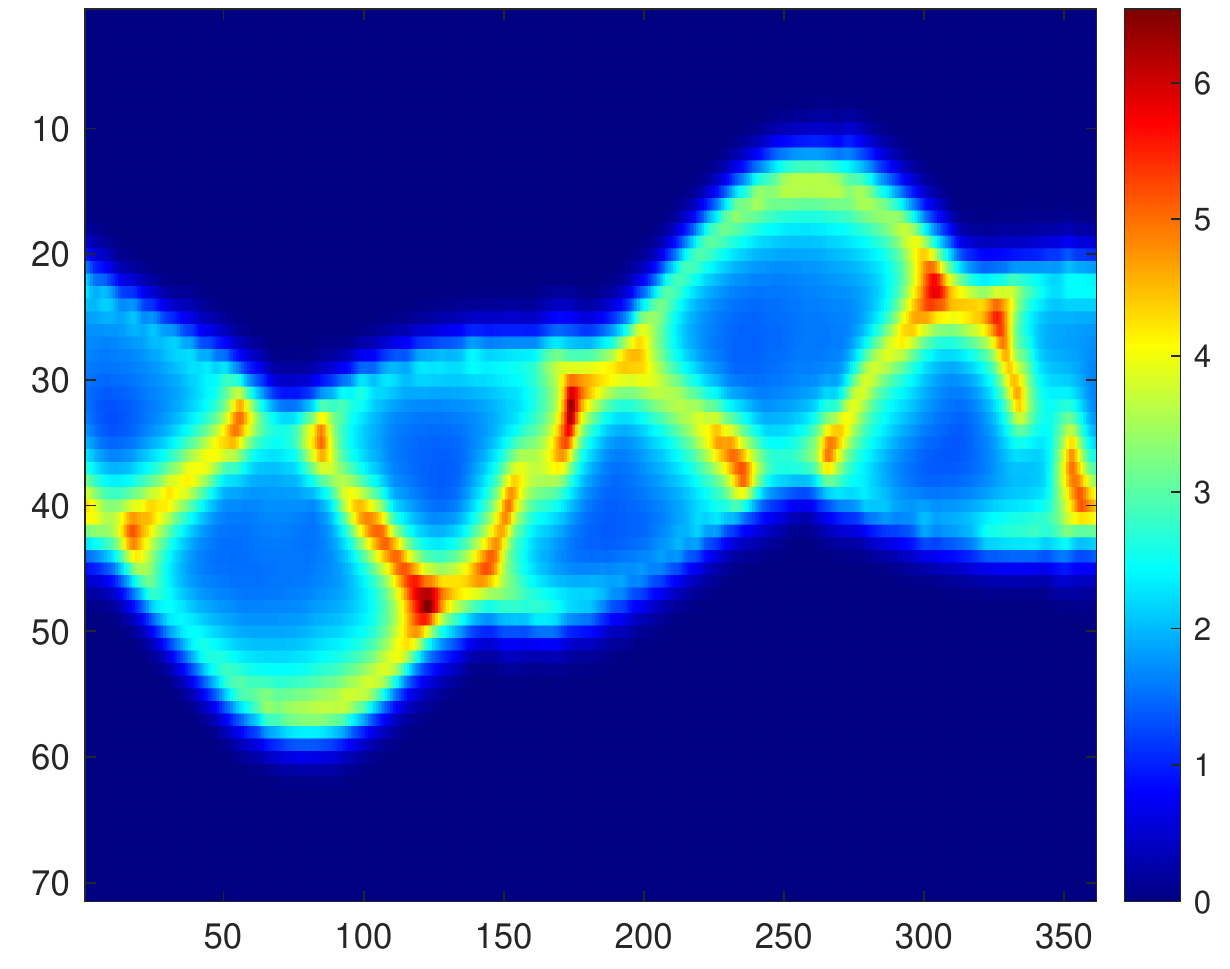}
    \qquad
    \includegraphics[scale=0.5]{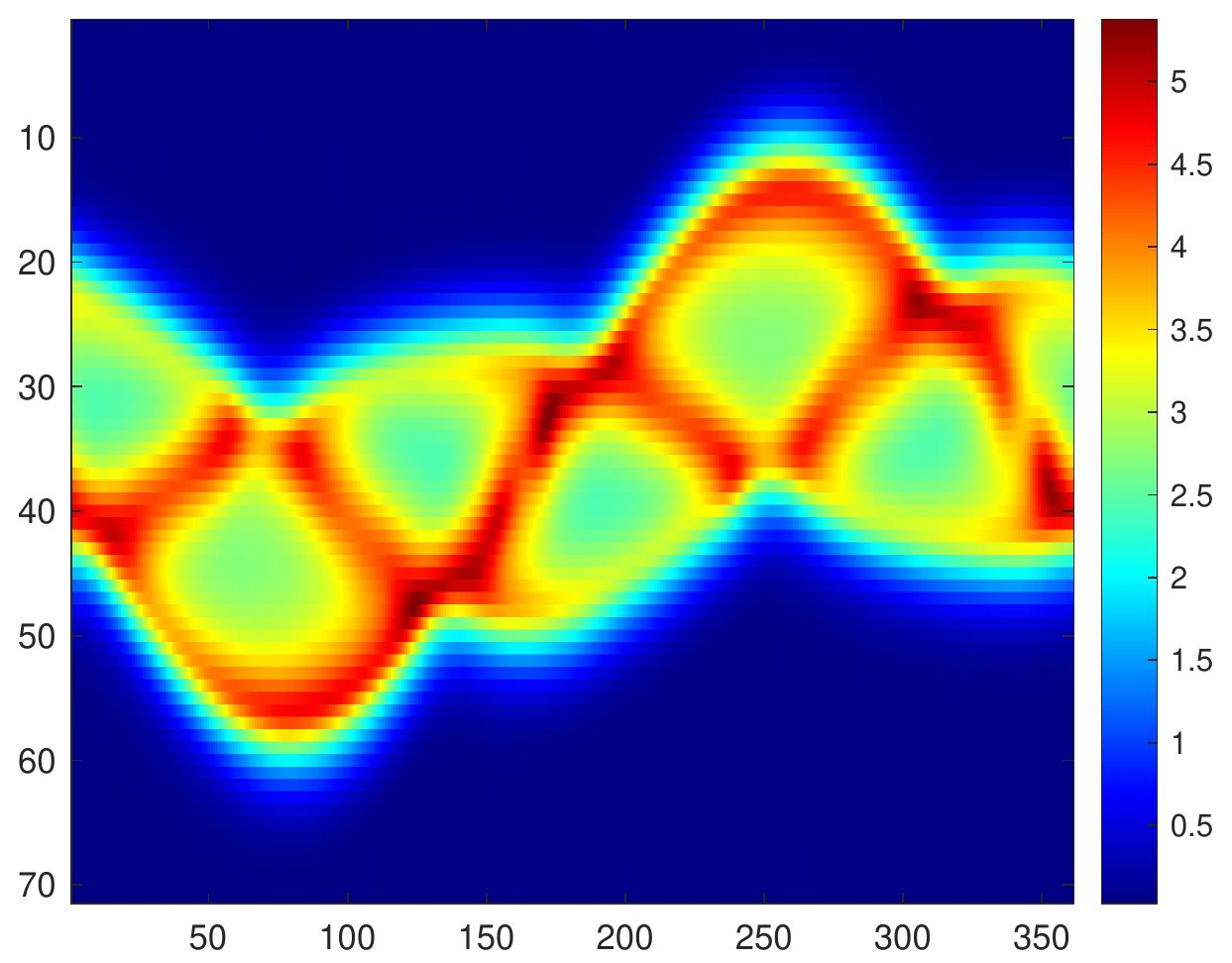}
    \caption{Pre-processed data $-2\log(\abs{P_\ij/\Pref})$ (left) and $-\log(I_\ij/\Iref)$ (right).}
    \label{fig_sinogram_P_I}
\end{figure}

Concerning the implementation of the different reconstruction approaches: the density function $f$ was discretized as a piecewise constant function on an $81 \times 81$ pixel grid, and for the assembly of the corresponding Radon transform matrix the AIR TOOLS II toolbox by Hansen and Jorgensen \cite{Hansen_Jorgensen_2017} was used. All computations were carried out in Matlab on a desktop computer with an Intel Xeon E-2136 processor with 3.30GHz and 16 GB RAM, and for the solution approach based on filtered back-projection, the Matlab function \emph{iradon} was used.

First, in order to check the validity of our general nonlinear model \eqref{eq_f_nonlin}, we set up a numerical representation of the density function of the triangular plastic sample shown in Figure~\ref{fig_plasticsample} (left). Applying the nonlinear operator $F$ defined in \eqref{def_F} to this, we obtain the simulated data depicted in Figure~\ref{fig_nonlinear_Landweber_sim} (left). Even though the numerical representation is only a coarse approximation of the real sample, and effects like scattering are not included in our model, the simulated data are in good agreement with the experimental data obtained from the actual THz measurement, which are shown in Figure~\ref{fig_nonlinear_Landweber_exp} (left). Hence, our nonlinear model appears to be a sufficiently accurate approximation of the actual underlying physical reality. The corresponding reconstructions via the nonlinear Landweber approach presented in Section~\ref{subsect_NonlinearReconstruction}, obtained after $200$ iterations from both the simulated and the experimental data, are depicted in Figure~\ref{fig_nonlinear_Landweber_sim} (right) and Figure~\ref{fig_nonlinear_Landweber_exp} (right), respectively. In both cases, we clearly recover the triangular structure of the sample, as well as the difference in thickness between the top edge of the triangle and its other two sides; compare also to Figure~\ref{fig_plasticsample} (left).

\begin{figure}[ht!]
    \centering
    \includegraphics[scale=0.5]{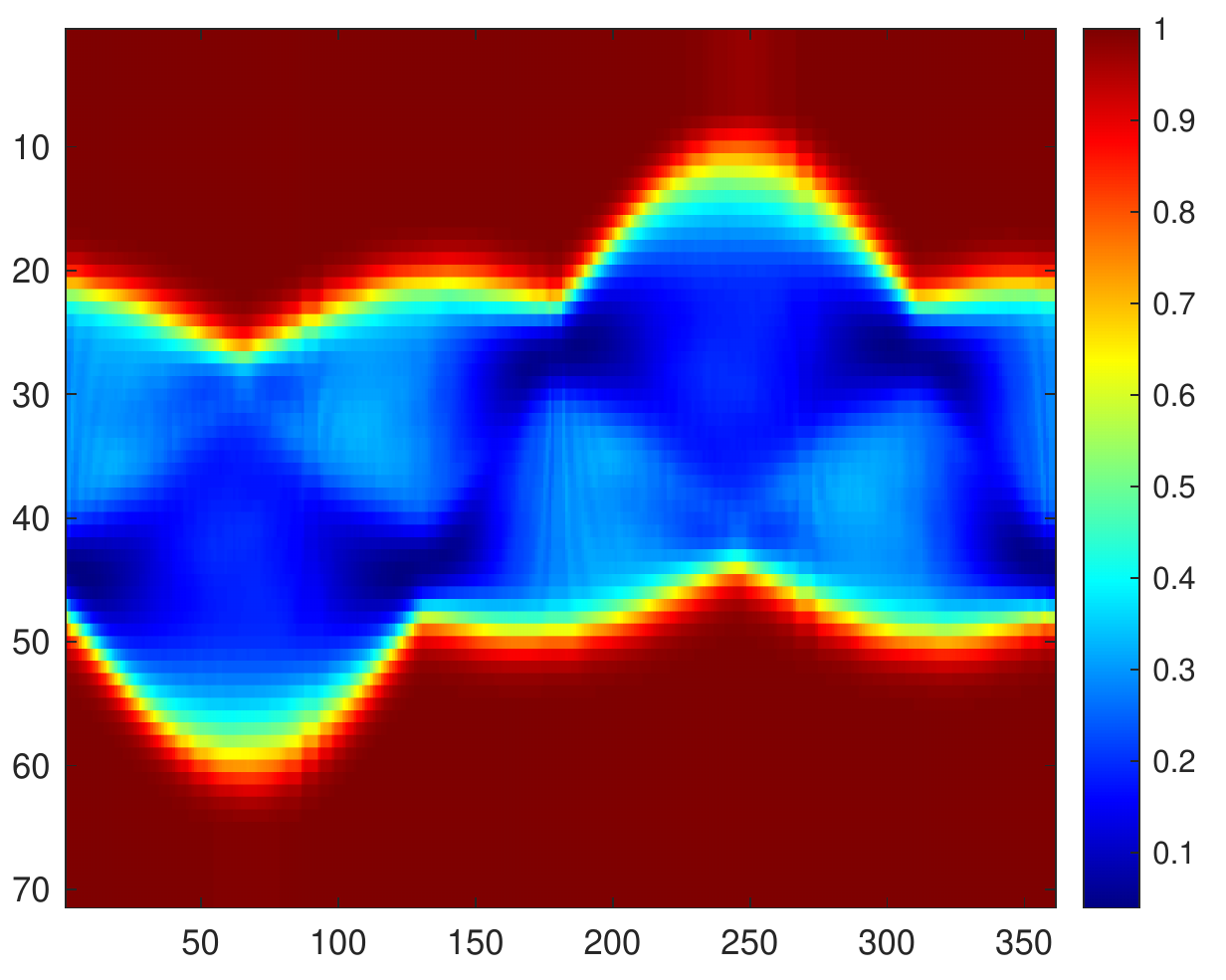}
    \qquad
    \includegraphics[scale=0.5]{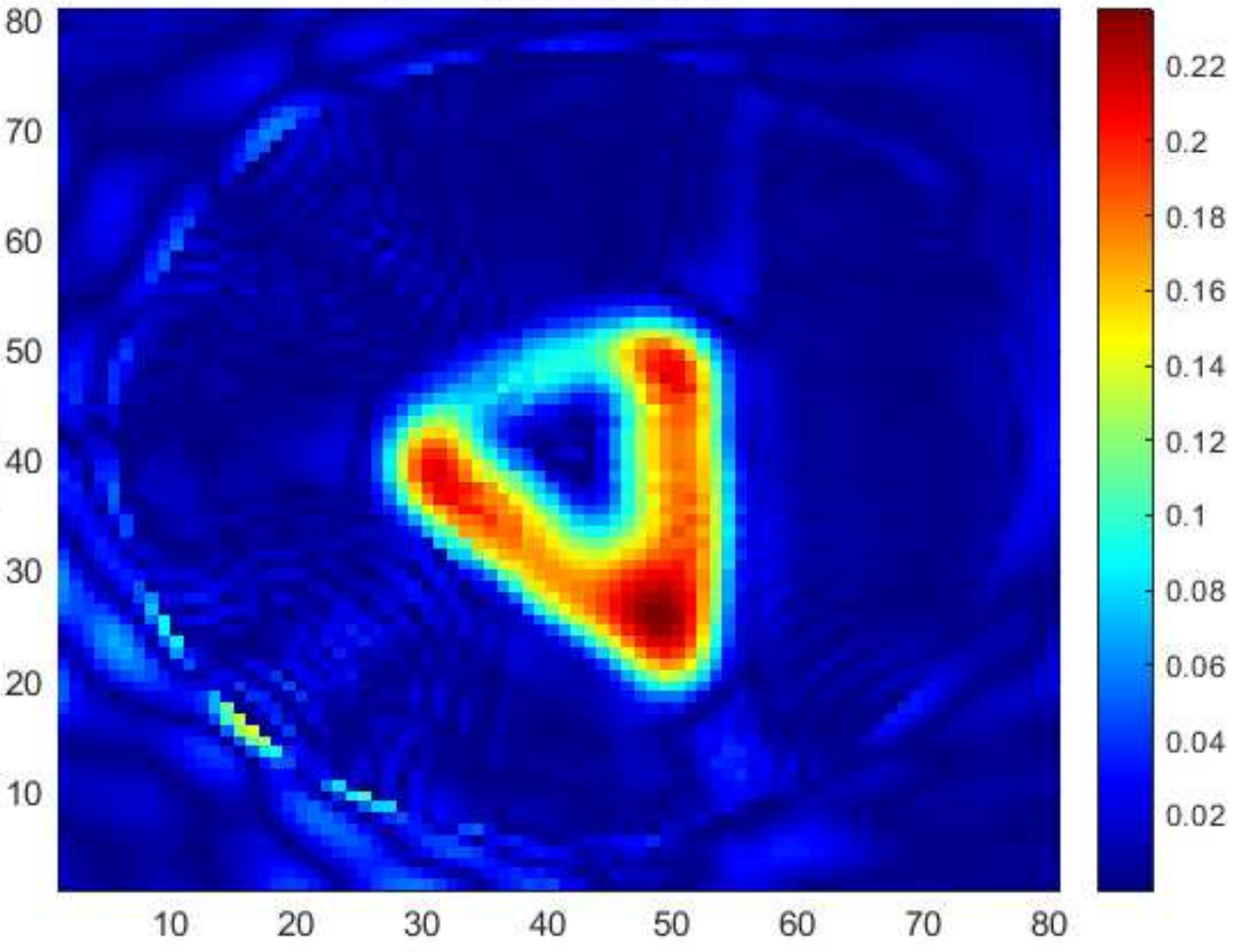}
    \caption{Simulated measurement data $\abs{P_\ij/\Pref}$ for the triangular plastic sample depicted in Figure~\ref{fig_plasticsample} (left), and the resulting reconstruction obtained via the nonlinear Landweber approach introduced in Section~\ref{subsect_NonlinearReconstruction} (right).}
    \label{fig_nonlinear_Landweber_sim}
\end{figure}

\begin{figure}[ht!]
    \centering
    \includegraphics[scale=0.5]{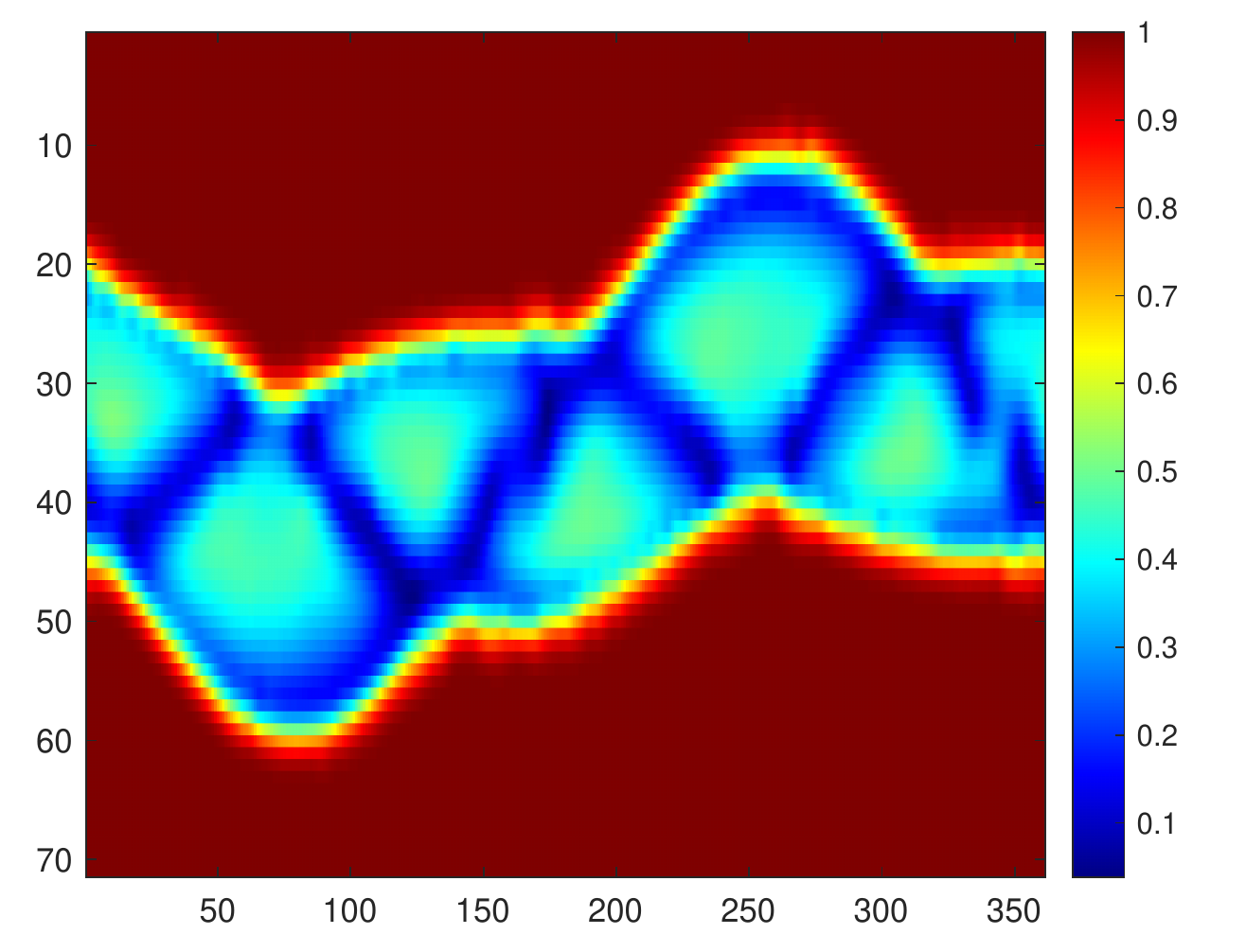}
    \qquad
    \includegraphics[scale=0.5]{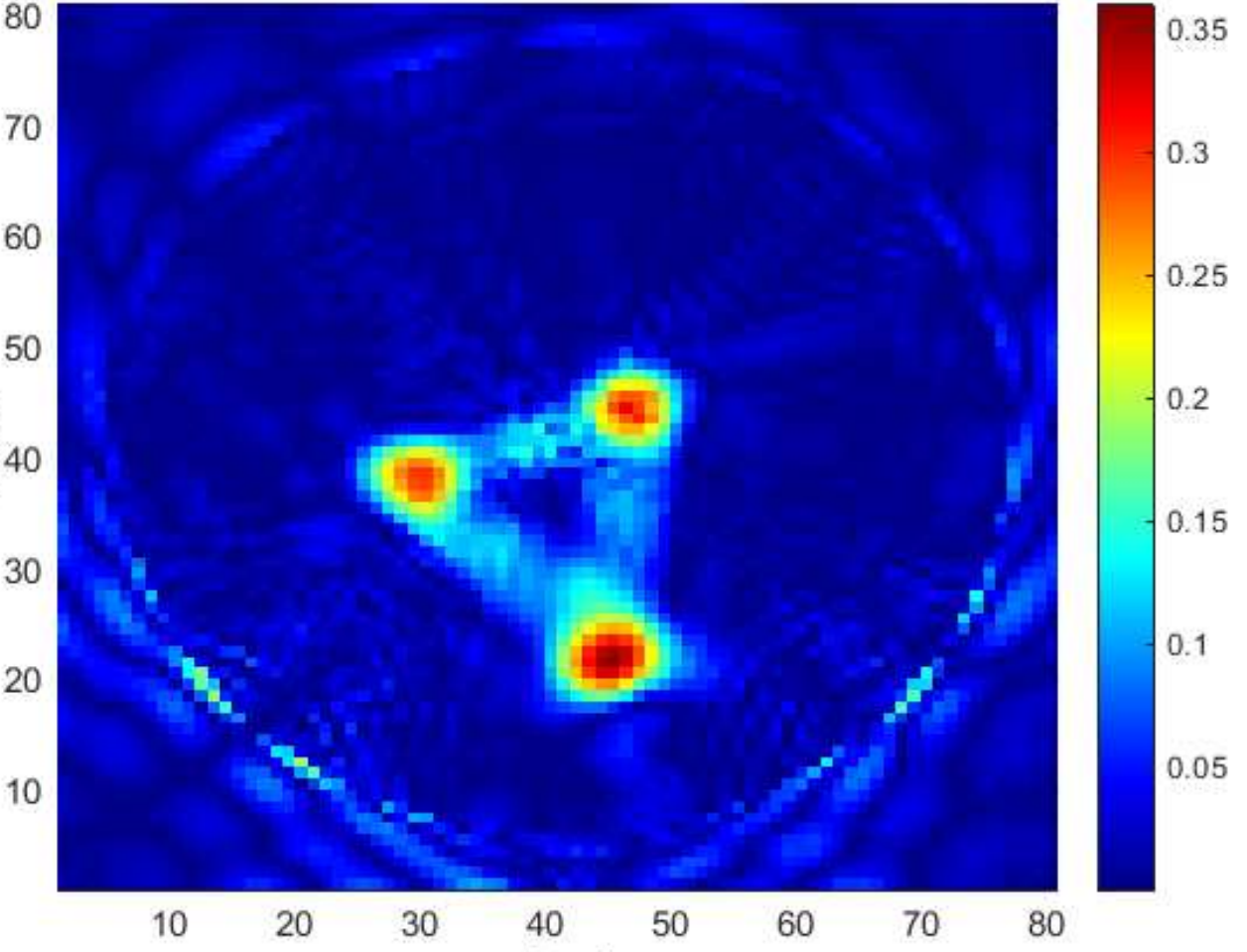}
    \caption{Pre-processed data $\abs{P_\ij/\Pref}$ obtained from THz measurements of the triangular plastic sample depicted in Figure~\ref{fig_plasticsample} (left), and the resulting reconstruction obtained via the nonlinear Landweber approach introduced in Section~\ref{subsect_NonlinearReconstruction} (right).}
    \label{fig_nonlinear_Landweber_exp}
\end{figure}

After considering the nonlinear Problem~\ref{prob_F} we now turn our attention to the linear Problems~\ref{prob_R_P} and \ref{prob_R_I}. In particular, we apply the different reconstruction approaches introduced in Section~\ref{subsect_LinearReconstruction}, i.e., filtered back-projection, contour tomography, Landweber iteration, and Tikhonov regularization, to the data $-2\log(\abs{P_\ij/\Pref})$ and $-\log(I_\ij/\Iref)$ depicted in Figure~\ref{fig_sinogram_P_I} in order to obtain solutions to  Problem~\ref{prob_R_P} and \ref{prob_R_I}, respectively. The resulting reconstructions are shown in Figure~\ref{fig_results_prob_R_P} and Figure~\ref{fig_results_prob_R_I}. Note that for Landweber iteration, a zero initial guess and $2000$ iterations were used, while for Tikhonov regularization we chose $\alpha = 500$ for the regularization parameter. In all reconstructions, the triangular structure of the sample is clearly recovered. Additionally, in the contour tomography reconstructions the outer and inner edges of the object are clearly visible. Note that all reconstructions feature more or less pronounced bumps in the corners of the triangle, which we suppose to be due to scattering effects not covered by our model. However, all reconstructions are sufficient to allow for a qualitative inspection of the internal structure of the measured plastic sample. Concerning the differences between the reconstructions, note that the reconstructions obtained via Problem~\ref{prob_R_P} feature somewhat sharper edges than those obtained via Problem~\ref{prob_R_I}.

\begin{figure}[ht!]
    \centering
    \includegraphics[trim = 0 0 0 18, clip, scale=0.5]{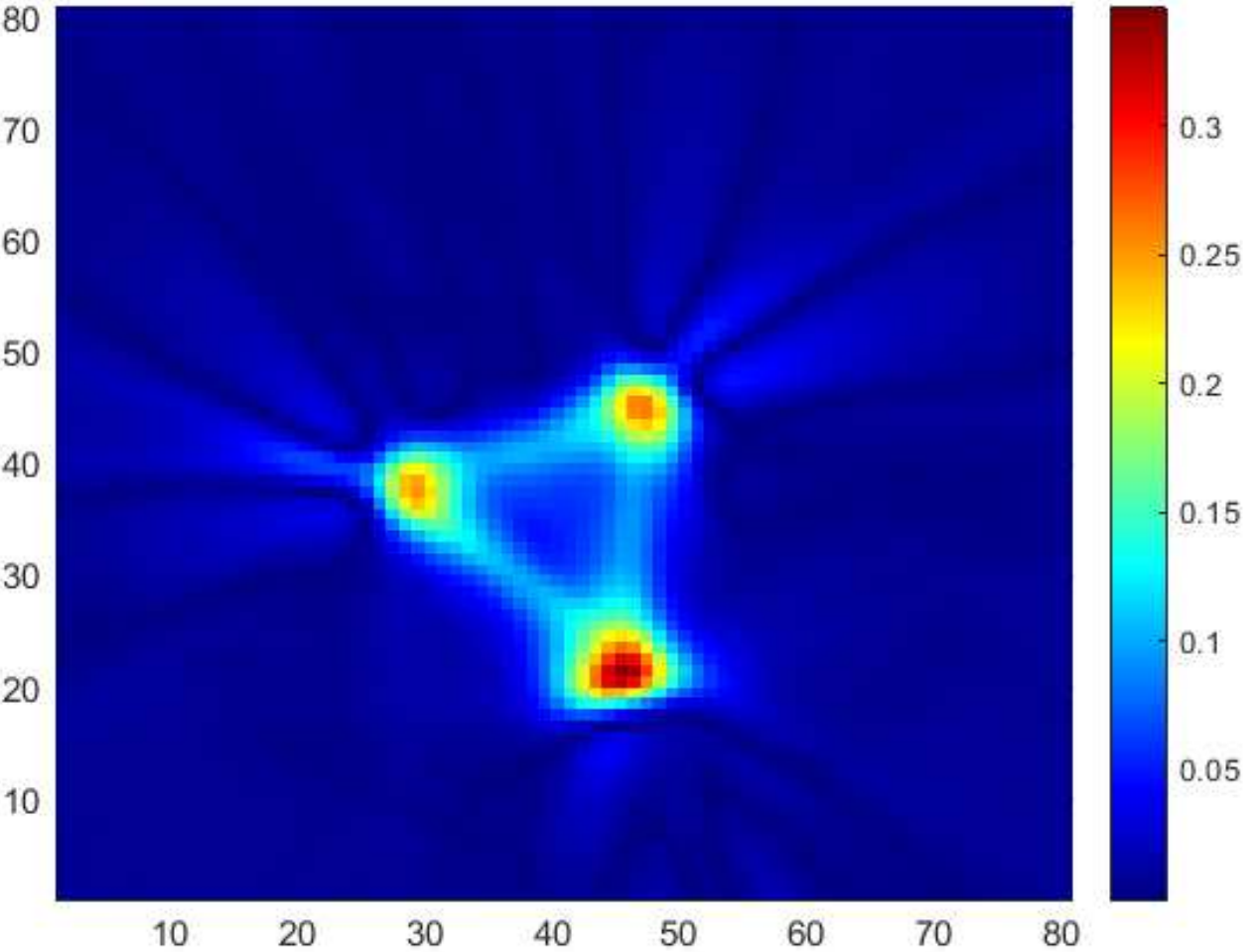}
    \qquad
    \includegraphics[trim = 0 0 0 17, clip, scale=0.5]{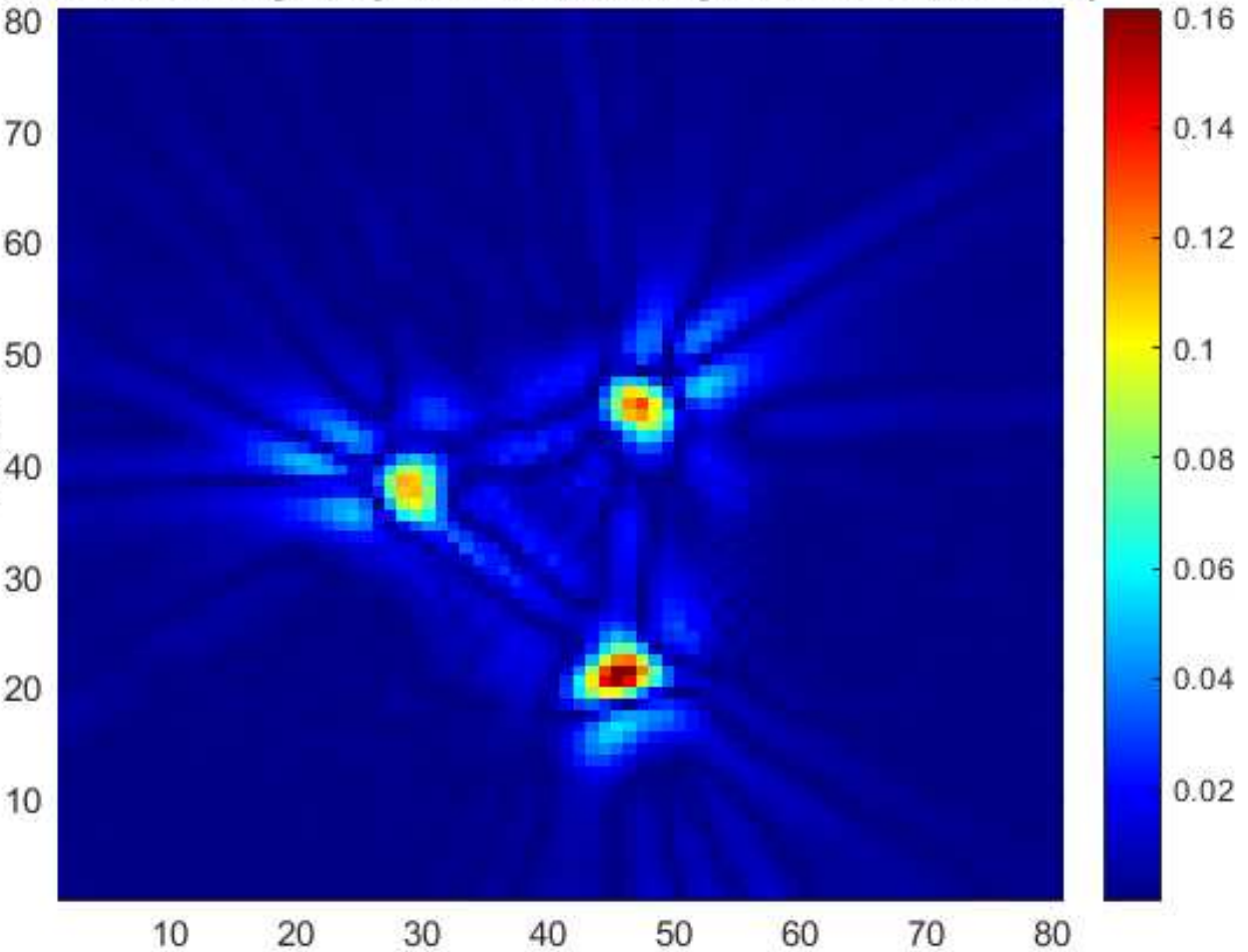}
    \\ \vspace{15pt}
    \includegraphics[trim = 0 0 0 21, clip, scale=0.5]{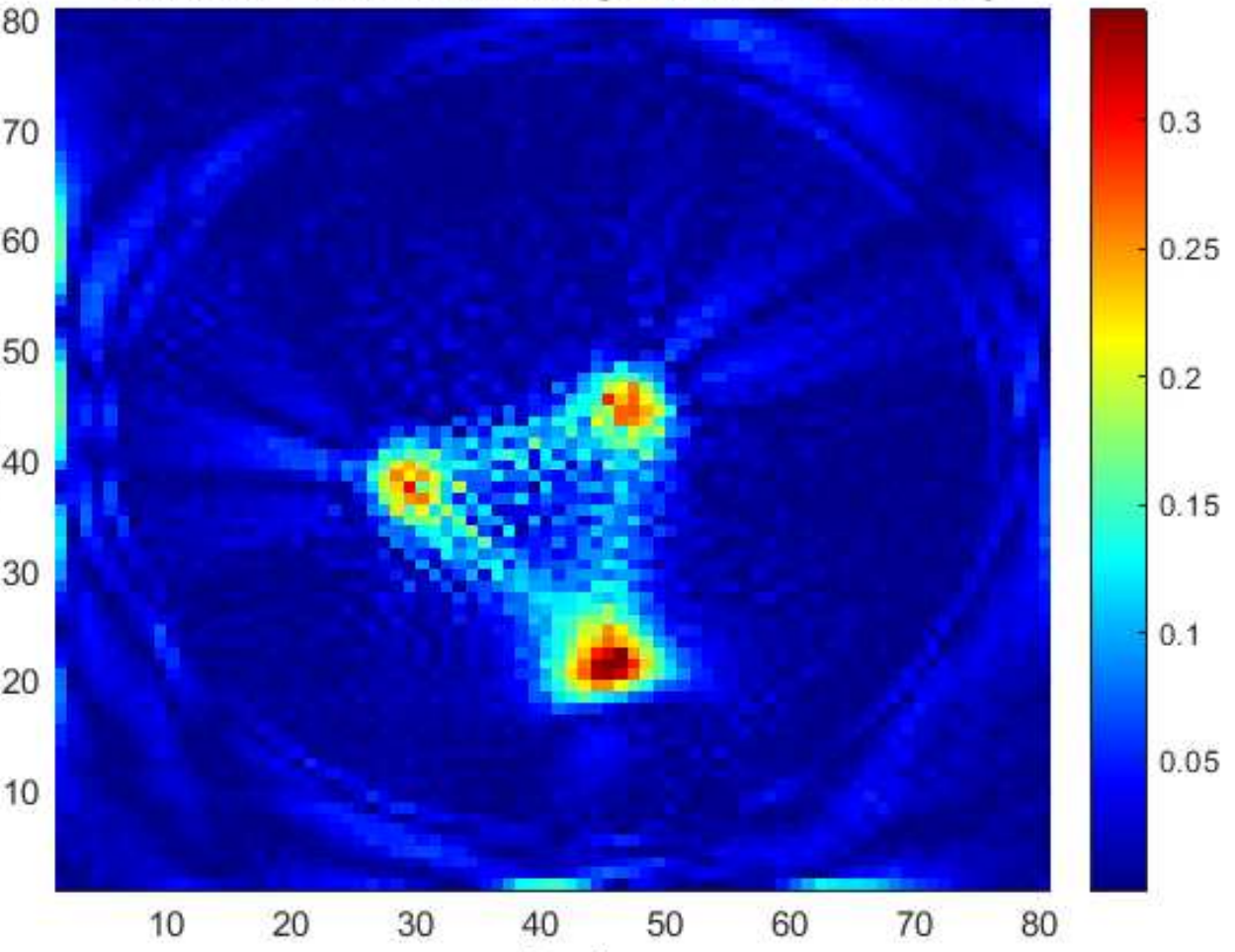}
    \qquad
    \includegraphics[trim = 0 0 0 17, clip, scale=0.5]{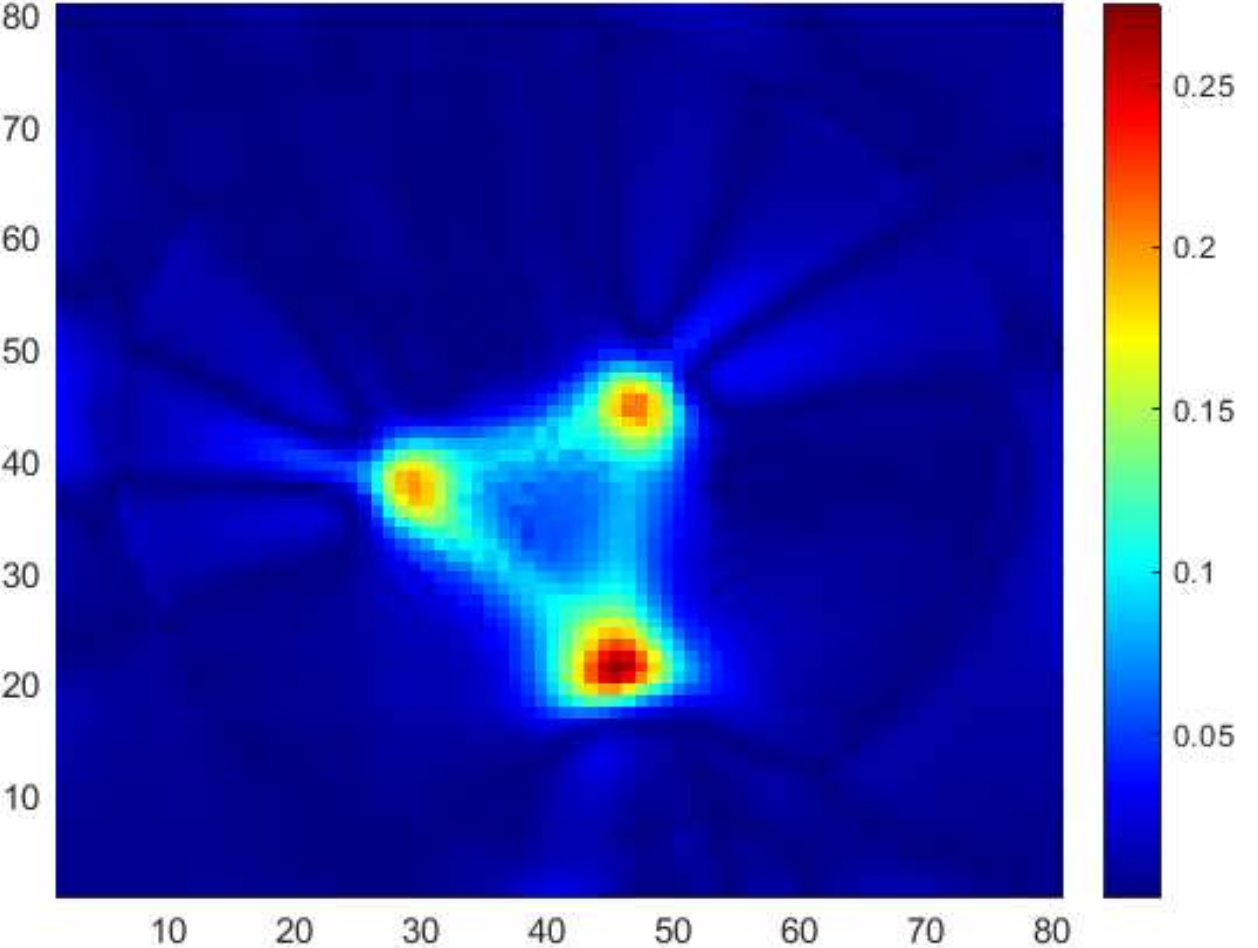}
    \caption{Reconstructions for Problem~\ref{prob_R_P} obtained from the data $-2\log(\abs{P_\ij/\Pref})$ depicted in Figure~\ref{fig_sinogram_P_I} (right) via the application of the following reconstruction methods introduced in Section~\ref{subsect_LinearReconstruction}: filtered back-projection (top left), contour tomography (top right), Landweber iteration (bottom left), Tikhonov regularization (bottom right).}
    \label{fig_results_prob_R_P}
\end{figure}

\begin{figure}[ht!]
    \centering
    \includegraphics[trim = 0 0 0 16, clip, scale=0.5]{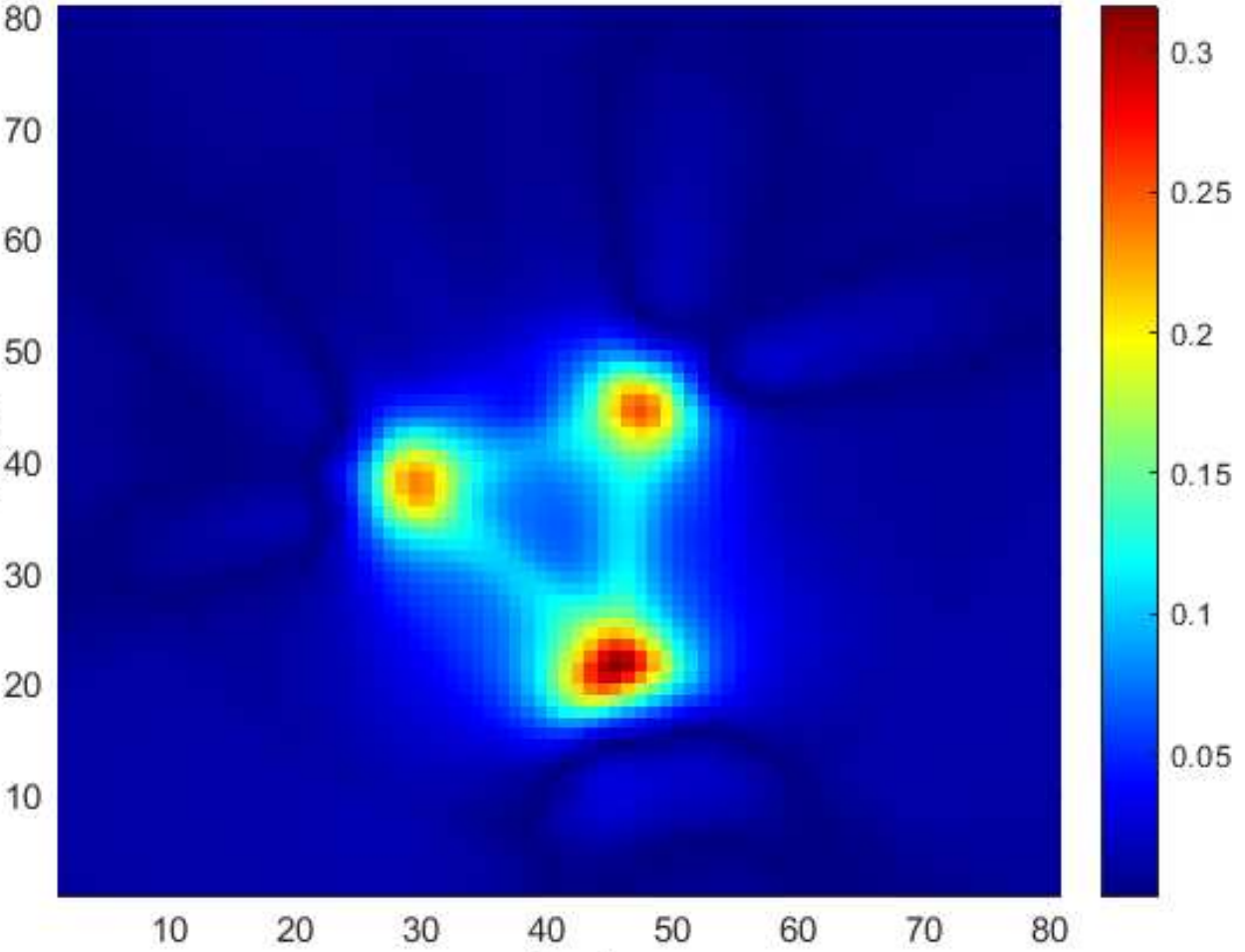}
    \qquad
    \includegraphics[trim = 0 0 0 17, clip, scale=0.5]{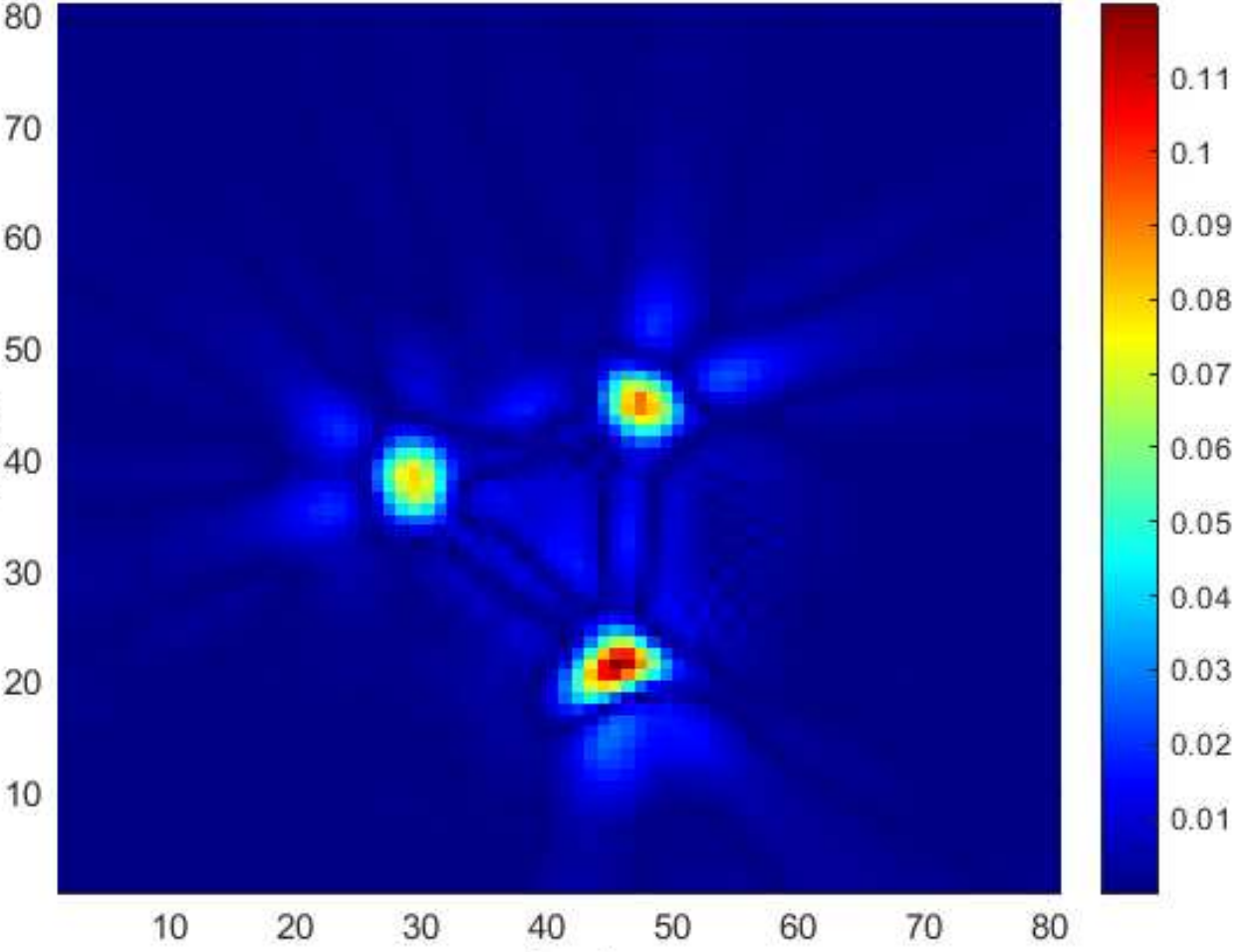}
    \\ \vspace{15pt}
    \includegraphics[trim = 0 0 0 19, clip, scale=0.5]{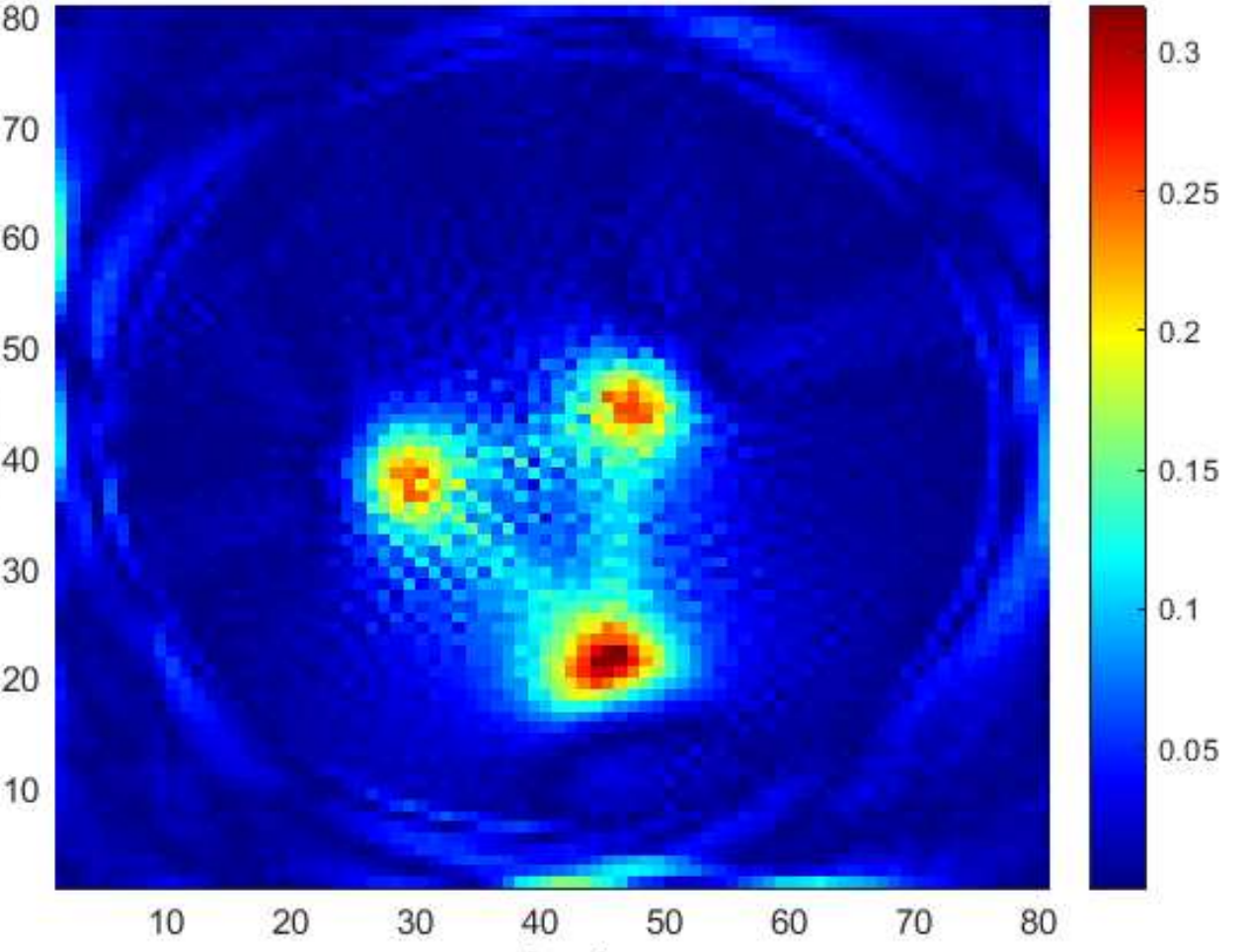}
    \qquad
    \includegraphics[trim = 0 0 0 17, clip, scale=0.5]{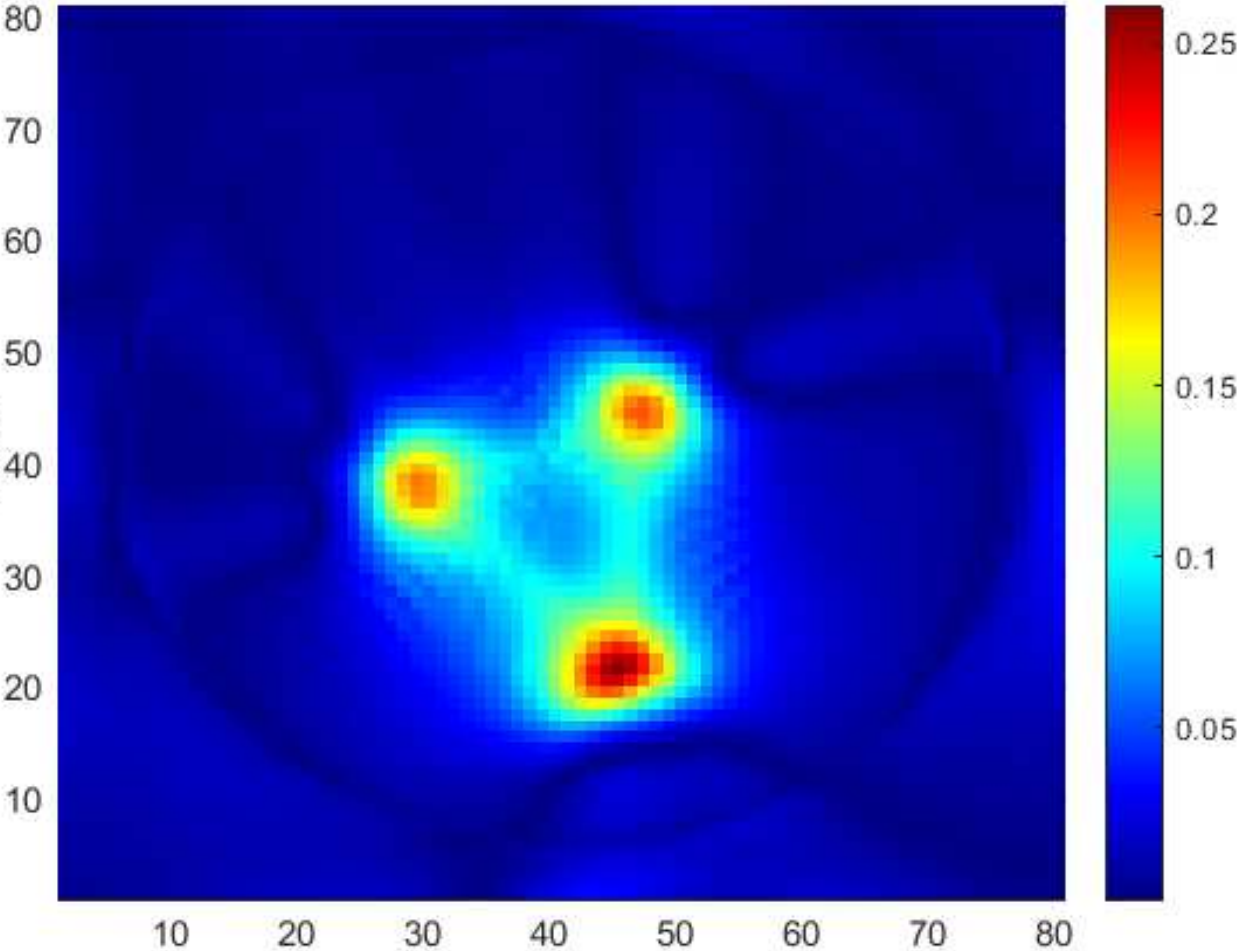} 
    \caption{Reconstructions for Problem~\ref{prob_R_I} obtained from the data $-\log(I_\ij/\Iref)$ depicted in Figure~\ref{fig_sinogram_P_I} (left) via the application of the following reconstruction methods introduced in Section~\ref{subsect_LinearReconstruction}: filtered back-projection (top left), contour tomography (top right), Landweber iteration (bottom left), Tikhonov regularization (bottom right).}
    \label{fig_results_prob_R_I}
\end{figure}

In summary, all of the presented reconstruction approaches were successful in recovering the triangular structure of the object. In addition, the reconstruction approach based on the nonlinear Problem~\ref{prob_F} is also able to resolve the difference in thickness between the top edge and the two side edges of the triangular sample. This indicates that the nonlinear model is closer to the physical reality, and that the corresponding nonlinear reconstruction approach can be beneficial in practice.

\section{Conclusion}\label{sect_conclusion}

In this paper, we considered the imaging problem of THz tomography, with an emphasis on the use case of THz imaging on plastic profiles via a THz-TDS system. In particular, we derived a nonlinear mathematical model describing this problem, and considered a number of linear approximations revealing connections to computerized tomography. Furthermore, we proposed different reconstruction approaches, which were numerically tested on experimental data obtained from THz measurements of a plastic sample.

\section{Support}

This project has received funding from the ATTRACT project funded by the EC under Grant Agreement 777222. Furthermore, financial support was provided by the Austrian research funding association (FFG) under the scope of the COMET programme within the research project “Photonic Sensing for Smarter Processes (PSSP)” (contract number 871974). This program is promoted by BMK, BMDW, the federal state of Upper Austria and the federal state of Styria, represented by SFG. S. Hubmer and R. Ramlau were (partly) funded by the Austrian Science Fund (FWF): F6805-N36. A. Ploier was also (partly) funded by the Austrian Science Fund (FWF): W1214-N15, project DK8. 

\bibliographystyle{plain}
{\small
\bibliography{mybib}

\begin{thebibliography}{10}

\bibitem{Dicken_1998}
V.~Dicken.
\newblock {\em {Simultaneous activity and attenuation reconstruction in single
  photon emission computed tomography, a nonlinear ill-posed problem}}.
\newblock PhD thesis, Universit\"at Potsdam, 1998.

\bibitem{Dicken_1999}
V.~Dicken.
\newblock {A new approach towards simultaneous activity and attenuation
  reconstruction in emission tomography}.
\newblock {\em Inverse Problems}, 15(4):931, 1999.

\bibitem{Dietz_Vieweg_ECOPS_2014}
R.~J.~B. Dietz, N.~Vieweg, T.~Puppe, A.~Zach, B.~Globisch, T.~Göbel,
  P.~Leisching, and M.~Schell.
\newblock {A}ll fiber-coupled {TH}z-{TDS} system with k{H}z measurement rate
  based on electronically controlled optical sampling.
\newblock {\em Optics letters}, 39:6482--5, 2014.

\bibitem{Engl_Hanke_Neubauer_1996}
H.~W. {Engl}, M.~{Hanke}, and A.~{Neubauer}.
\newblock {\em {Regularization of inverse problems.}}
\newblock Dordrecht: Kluwer Academic Publishers, 1996.

\bibitem{Garcea_Wang_Withers_XCTforComposites_2018}
S.~C. Garcea, Y.~Wang, and P.J.Withers.
\newblock X-ray computed tomography of polymer composites.
\newblock {\em Composites Science and Technology}, 156(1):305--319, 2018.

\bibitem{Griffiths_Electrodynamics_2014}
D.~J. Griffiths.
\newblock {\em {I}ntroduction to electrodynamics}.
\newblock Pearson, 4th edition, 2014.

\bibitem{Hansen_Jorgensen_2017}
P.~C. Hansen and J.~Jorgensen.
\newblock Air tools ii: algebraic iterative reconstruction methods, improved
  implementation.
\newblock {\em Numerical Algorithms}, 79, 11 2017.

\bibitem{Jin_Kim_Jeon_THzDielectricConstants_2006}
Y.~Jin, G.~Kim, and S.~Jeon.
\newblock {T}erahertz {D}ielectric {P}roperties of {P}olymers.
\newblock {\em Journal of the Korean Physical Society}, 49(2):513--517, 2006.

\bibitem{Kaltenbacher_Neubauer_Scherzer_2008}
B.~{Kaltenbacher}, A.~{Neubauer}, and O.~{Scherzer}.
\newblock {\em {Iterative regularization methods for nonlinear ill-posed
  problems.}}
\newblock Berlin: de Gruyter, 2008.

\bibitem{Louis_1989}
A.~K. Louis.
\newblock {\em {I}nverse und schlecht gestellte {P}robleme}.
\newblock Teubner Studienb{\"u}cher Mathematik. Vieweg+Teubner Verlag, 1989.

\bibitem{Louis_Maass_1993}
A.~K. Louis and P.~Maass.
\newblock {Contour Reconstruction in 3-D X-Ray CT}.
\newblock {\em IEEE Transactions on Medical Imaging}, 12(4), 1993.

\bibitem{Mukherjee_Federici_pulsedTHzCT}
S.~Mukherjee and J.~Federici.
\newblock Study of structural defects inside natural cork by pulsed terahertz
  tomography.
\newblock In {\em 2011 International Conference on Infrared, Millimeter, and
  Terahertz Waves}, pages 1--2, 2011.

\bibitem{Natterer_2001}
F.~{Natterer}.
\newblock {\em {The Mathematics of Computerized Tomography}}.
\newblock Society for Industrial and Applied Mathematics, 2001.

\bibitem{Neu_Schmuttenmaer_THzTDS_2018}
J.~Neu and C.~A. Schmuttenmaer.
\newblock {T}utorial: {A}n introduction to terahertz time domain spectroscopy
  ({TH}z-{TDS}).
\newblock {\em Journal of Applied Physics}, 124(23):231101, 2018.

\bibitem{Recur_cwTHzCT_2012}
B.~Recur, J.~P. Guillet, L.~Bassel, C.~Fragnol, I.~Manek-Hönninger,
  J.~Delagnes, W.~Benharbone, P.~Desbarats, J.~Domenger, and P.~Mounaix.
\newblock {T}erahertz radiation for tomographic inspection.
\newblock {\em Optical Engineering}, 51(9):1--8, 2012.

\bibitem{Scherzer_1996}
O.~{Scherzer}.
\newblock A convergence analysis of a method of steepest descent and a
  two{-}step algorithm for nonlinear ill{-}posed problems.
\newblock {\em Numerical Functional Analysis and Optimization},
  17(1-2):197--214, 1996.

\bibitem{Yahyapour_Jahn_ThicknesswECOPS_2019}
M.~Yahyapour, A.~Jahn, K.~Dutzi, T.~Puppe, P.~Leisching, B.~Schmauss,
  N.~Vieweg, and A.~Deninger.
\newblock {F}astest {T}hickness {M}easurements with a {T}erahertz
  {T}ime-{D}omain {S}ystem based on {E}lectronically {C}ontrolled {O}ptical
  {S}ampling.
\newblock {\em Applied Sciences}, 9(7):1283, 2019.

\end{thebibliography}
}

\end{document}